\documentclass[12pt]{article}
\usepackage{DLSCQP}

\title{Team Semantics and Independence Notions in Quantum Physics}
\author{Samson Abramsky\thanks{Partially supported by  U.K. EPSRC Research Fellowship EP/V040944/1, Resources in Computation.}\\
Department of Computer Science,\\ University College London, England\and 
Joni Puljujärvi\thanks{Partially supported by the Academy of Finland, grant 322795. This project has received funding from the European Research Council (ERC) under the European Union’s Horizon 2020 research and innovation programme (grant agreement No 101020762).}\\ Department of Mathematics and Statistics,\\ University of Helsinki, Finland \and 
Jouko Väänänen\footnotemark[2]\\ Department of Mathematics and Statistics,\\ University of Helsinki, Finland}
\date{\today}

\newcommand{\vx}{{\vec{x}}}
\newcommand{\vy}{{\vec{y}}}
\newcommand{\vz}{{\vec{z}}}
\newcommand{\vu}{{\vec{u}}}
\newcommand{\vv}{{\vec{v}}}
\newcommand{\vw}{{\vec{w}}}
\newcommand{\va}{{\vec{a}}}
\newcommand{\vb}{{\vec{b}}}
\newcommand{\vc}{{\vec{c}}}

\newcommand{\ie}{\textit{i.e.}~}

\newcommand{\BB}{\mathbb{B}}

\newcommand{\system}{\mathcal{S}}

\begin{document}

\maketitle

\begin{abstract}
    We study dependence and independence concepts found in quantum physics, especially those related to hidden variables and non-locality, through the lens of team semantics and probabilistic team semantics, adapting a relational framework introduced in~\cite{MR3038040}. This also leads to new developments  in independence logic and probabilistic independence logic.
\end{abstract}

\section{Introduction}

The semantics of first-order logic is based on the inductively defined concept of an assignment satisfying a  formula in a given model. In a more general approach, called \emph{team semantics}, the basic concept is that of a \emph{set of assignments} satisfying a formula in a  model. This allows consideration of new atomic formulas such as ``$x$ is totally determined by $y_1,\ldots,y_n$'' and ``$x_1,\ldots,x_n$ are independent of $y_1,\ldots,y_m$''.
Such constraints on variables appear throughout sciences but in experimental sciences in particular. In this paper we apply team semantics to investigate determinism and independence concepts in quantum physics, following very closely~\cite{MR3038040}. In an independent development, R.~Albert and E.~Gr\"adel have in their paper~\cite{2021arXiv210210931A} come to many of the same conclusions.

The indeterministic and non-local nature of quantum mechanics, since its conception, has challenged the deterministic, local view of the world. To retain a more classical looking picture, several hidden-variable models for quantum mechanics---that would explain quantum  behaviour in terms of 
an underlying  local and deterministic theory---have been proposed since the 1920s. These models try to explain the predictions of quantum mechanics by adding unobservable hidden variables that play a role in determining the state of a quantum system. And indeed, if no constraints are posed on how the hidden variables can act---for instance, if the hidden variables are allowed to influence which measurements we make---then we can certainly come up with a hidden-variable explanation of anything. However, in order to form a reasonable and satisfactory theory, one needs to require that the hidden-variable models satisfy some combination of natural properties such as \emph{Bell locality}. A critical challenge for the hidden-variable program then emerged in the form of the famous no-go theorems by Bell and others~\cite{MR3790629, PhysRev.47.777, MR1081993, PhysRevLett.71.1665, MR0219280}: they showed that models satisfying what are generally regarded as reasonable assumptions could provably never account to the predictions of quantum mechanics.

\Samson introduced in~\cite{MR3038040} a relational framework for developing the key notions and results on hidden variables and non-locality, which can be seen as a relational variant of the probabilistic setting of~\cite{MR2443068}. He introduced what he called ``relational empirical models'' and used them to show that the basic results of the foundations of quantum mechanics, usually formulated in terms of probabilistic models, can be seen already on the level of mere (two-valued) relations. Our key observation is that we can think of the relational empirical models of~\cite{MR3038040} as \emph{teams} in the sense of team semantics. The basic quantum-theoretic properties of relational empirical models can then be defined in terms of the independence atoms of independence logic~\cite{MR3038039}. We show that the relationships between quantum-theoretic properties of relational models become instances of logical consequence of independence logic in its team semantics. In fact, the existential-positive-conjunctive fragment suffices. The no-go theorems become instances of failure of logical consequence between specific formulas of independence logic. This extends also to probabilistic models, with independence logic replaced by the probabilistic independence logic of~\cite{MR3802381}, capturing the probabilistic notions of~\cite{MR2443068}.

Logical consequence in independence logic is, in general, non-axiomatizable. Even on the level of atoms no finite axiomatization exists~\cite{MR662613}. This shows that the concept of logical consequence is here highly non-trivial and potentially quite complex. It should be emphasised that the logical consequences arising from the quantum-theoretic examples are purely logical, having a priori nothing to do with quantum mechanics, and hence they apply to any other field where independence plays a role, e.g. the theory of social choice or biology. On the other hand, \samson introduces in~\cite{MR3038040} a concept which in team semantics characterizes those teams which can arise from  quantum-mechanical experiments. Presumably the most subtle relationships between quantum-mechanical concepts are particular to such quantum-theoretic teams. We introduce to probabilistic independence logic, expanding on the example of~\cite{MR3038040}, the concept of being \emph{finite-dimensional tensor-product quantum-mechanical} and propose questions it gives rise to.

We think that translating~\cite{MR3038040} to the language and terminology of team semantics is interesting in itself from the point of view of team semantics. However, our paper goes beyond this. We use the language of independence logic and probabilistic independence logic to express hidden-variable properties of empirical models and probabilistic empirical models. This calls for some new developments in independence logic itself. For example,
we use the existential quantifier of independence logic to guess values of hidden variables, but since the values may be outside the current domain, we introduce to independence logic the existential quantifier of sort logic~\cite{MR3205075}, which  allows the extension of the domain by new sorts. 

Relations between hidden-variable properties can be seen as logical consequences in independence logic. In some cases these logical consequences are provable from the axioms.
We use probabilistic independence logic to express probabilistic hidden-variable properties and their mutual relationships. We prove the probabilistic validity of axioms and rules of independence logic, so the relationships of probabilistic hidden-variable models that follow from the axioms of independence logic hold also probabilistically.
We introduce an operator $\PR\varphi$ which holds in a team if and only if the team is the possibilistic collapse of a probabilistic team satisfying $\varphi$. Adopting the concept of a quantum realizable team from~\cite{MR3038040} we introduce the operator $\QR\varphi$ which holds in a team if and only if the team is the possibilistic collapse of a probabilistic  team that satisfies $\varphi$ and whose probability distribution arises from a finite-dimensional quantum system. We take the first step towards developing independence logic with the operators $\PR$ and $\QR$.

This paper is part of a program to find general principles that govern the uses of dependence and independence concepts in science and humanities.

We are grateful to Philip Dawid, Miika Hannula, Åsa Hirvonen, Martti Karvonen, and Juha Kontinen, among others, for useful discussions and remarks.

\section{Dependence and Independence Logic}\label{dail}

The basic concept of the semantics of first-order logic is that of an \emph{assignment}, \ie an assignment of values in the universe of a structure to a set of variables. This allows meaning to be assigned to  formulas with free variables, and hence enables a compositional definition of the semantics of formulas, with the truth conditions for sentences as a special case.
The concept of a \emph{team}, i.e. a set of assignments,  was introduced in~\cite{MR2351449} to make sense of the dependence atom $\=(\vec{x},\vec{y})$, ``$\vec{x}$ totally determines $\vec{y}\,$''. 
The meaning of the dependence atom $\=(\vec{x},\vec{y})$ of~\cite{MR2351449} in a team $X$ is
\[
    \forall s,s'\in X(s(\vec{x})=s'(\vec{x}) \implies s(\vec{y})=s'(\vec{y})).
\]
Our starting point in this paper is the observation that teams arise naturally in describing the kinds of situations which are the subject of Bell-type non-locality theorems.
We shall consider systems which have $n$ parties. If $n = 2$, we have bipartite systems. The parties are typically referred to as Alice, Bob, etc.
The physical idea behind this is that the parties are spacelike separated; hence, for the physical events under consideration, under relativistic constraints there is no possibility for information to pass between the parties.
We now consider the scenario where each party performs a measurement. Each such measurement has an \emph{input} (often referred to as the measurement \emph{setting}), and an \emph{output} (often referred to as the measurement \emph{outcome}). 
The input could be turning a knob to a certain position, choosing the angle of a magnetic field, etc. The output of the measurement could be ``true'' or ``false'' corresponding to the presence or absence of a click in a detector, a reading of a gauge, etc.

Let us consider as an example the famous Stern--Gerlach experiment~\cite{Gerlach1922} which was one of the early experiments manifesting quantization, here quantization of angular momentum. In this experiment a beam of silver atoms is directed through a sequence of magnets towards a detector screen. Although the silver atoms are not electrically charged, quantum theory, unlike classical physics, predicts that the atoms are deflected by the magnets. In this experiment the orientations of the magnets are what we call the measurements. The coordinates of the points of collision of the atoms with the detector screen are what we call the outcomes. As it happened, the experiment showed in 1922 clearly that the coordinates manifest quantization of the deflection angle.

A single  event can be represented using a variable $x_i$ for each input and a variable $y_i$ for each output. Such a single-shot event is then  represented by assignment of measurement settings to the inputs, and outcomes to the outputs. This is just an assignment to the set of variables $\{ x_0, \ldots , x_{n-1}, y_0, \ldots , y_{n-1}\}$.
We are interested in \emph{ensembles} of such events, which allow non-deterministic and probabilistic variation in the outcomes of given measurements to be captured.
Operationally, such ensembles can be generated by repeatedly performing multipartite measurements, and recording the outcomes.
On the quantitative level, this will generate statistics, which can be represented by probability distributions on these events. We will look at this quantitative level later in the paper, but for now, we focus on qualitative information at the possibilistic level: do certain outcomes for given measurements ever arise? This information can be represented by the set of possible assignments, which will have the following form:

\begin{center}
    $X=$
    \begin{tabular}{|ccccc|}
        $x_0\phantom{{}^{-1}}$    & $y_0\phantom{{}^{-1}}$    & $\ldots$ & $x_{n-1}$       & $y_{n-1}$       \\
        \hline
        $a^0_0\phantom{{}^{-1}}$  & $b^0_0\phantom{{}^{-1}}$  & $\ldots$ & $a^0_{n-1}$     & $b^0_{n-1}$     \\
        $a^1_0\phantom{{}^{-1}}$  & $b^1_0\phantom{{}^{-1}}$  & $\ldots$ & $a^1_{n-1}$     & $b^1_{n-1}$     \\
        $\vdots\phantom{{}^{-1}}$ & $\vdots\phantom{{}^{-1}}$ & $\ddots$ & $\vdots$        & $\vdots$        \\
        $a^{m-1}_0$               & $b^{m-1}_0$               & $\ldots$ & $a^{m-1}_{n-1}$ & $b^{m-1}_{n-1}$ \\
    \end{tabular}\,.
\end{center}

We can think of $X$ as a \emph{team} (in the sense of team semantics) consisting of assignments of values to the variables $x_0,\dots,x_{n-1},y_0,\dots,y_{n-1}$. Even though the data in its intended interpretation has a clear structure dividing the elements of the table into ``inputs'' and ``outputs'', we can also look at the table as a mere database of data irrespective of how it was created. We can ask what kind of dependences this table of data---team---manifests.

Thus we can say that the team of data $X$  supports \textbf{strong determinism} if it satisfies
\[
    \=(x_i,y_i)
\]
for all $i<n$. Intuitively, in each such experiment the input for the $i$'th party completely determines the outcome for that party, that is, the $i$'th outcome does not, in the light of $X$, depend on anything other than the $i$'th input. This is a very strong constraint, which limits the applicability of this concept. 

We say that the team $X$ supports \textbf{weak determinism} if it satisfies
\[
    \=(x_0,\ldots,x_{n-1},y_i)
\]
for all $i<n$. Intuitively, that says that the whole set of inputs to the system collectively completely determine each outcome, that is, the outcome does not, in the light of $X$, depend on anything else than the inputs of the system. In systems arising from scientific experiments this means that the system has enough ``variables'' to determine its outcome. 

Consider the Stern--Gerlach experiment. Even if the magnets are directed in the same way, the particles that pass through the magnetic field manifest a (quantized) spectrum of results, rather than a single spot on the receiving screen. In keeping with a fundamental tenet of quantum physics, all phenomena such as tested by the Stern--Gerlach experiment give only probabilistic results. An individual team may not reveal this, but the bigger the team, the more likely it is to fail to support even weak determinism.

There are important aspects of experimental data that cannot be expressed in terms of the dependence atom only. We therefore move on to a stronger concept, one that supersedes dependence and allows to express also independence.

In independence logic~\cite{MR3038039}, we add a new atomic formula
\[
    \vec{y}\perp_{\vec{x}}\vec{z}
\]
to first-order logic. Intuitively this formula says that keeping $\vec{x}$ fixed, $\vec{y}$ and $\vec{z}$ are independent of each other. A team $X$ is defined to satisfy $\vec{y}\perp_{\vec{x}}\vec{z}$ if
\begin{gather*}
    \forall s,s'\in X[s(\vec{x})=s'(\vec{x})\implies \\
    \exists s''\in X(s''(\vec{x})=s(\vec{x})\wedge s''(\vec{y})=s(\vec{y})\wedge s''(\vec{z})=s'(\vec{z})].
\end{gather*}
We may observe that, unlike $\=(x_0,\dots,x_{n-1},y_i)$, this is not closed downwards\footnote{A formula $\varphi$ is closed downwards if, whenever $X\models\varphi$ and $Y\subseteq X$, we have $Y\models\varphi$.}, but it is closed under unions of increasing chains. Note that this condition is first order, as was the case for the semantics of the dependence atom. Thus independence logic is  $\Sigma^1_1$ in its expressive power, and hence NP.
Here is an example of a team satisfying $y_0 \perp_{x_0x_1} y_1$:
\begin{center}
    $X=$
    \begin{tabular}{|cccc|}
        $x_0$ & $y_0$ & $x_1$ & $y_1$ \\
        \hline
        $0$   & $1$   & $0$   & $1$   \\
        $0$   & $1$   & $1$   & $2$   \\
        $0$   & $1$   & $1$   & $7$   \\
        $0$   & $5$   & $1$   & $2$   \\
        $0$   & $5$   & $1$   & $7$   \\
        $1$   & $5$   & $1$   & $1$   \\
    \end{tabular}
\end{center}

For fixed $x_0$ and $x_1$, e.g. $x_0=0, x_1=1$, the values of $y_0$ and $y_1$ are independent of each other in the strong sense that if a value of $y_0$ occurs in combination with \emph{any} value of $y_1$, e.g. 2, it occurs also with any other value of $y_1$, e.g. 7. Intuitively this says that in these experiments the individual experiments do not interfere with each other. 
It is like measuring commuting quantum observables.

Note that the dependence atom can be defined in terms of the independence atom:
\[
    \=(\vec{x},\vec{y}) \equiv \vec{y} \perp_{\vec{x}} \vec{y}.
\]
We will thus use $\=(\vec{x},\vec{y})$ as a shorthand for $\vec{y} \perp_{\vec{x}} \vec{y}$ when dealing with independence logic.

\subsection{Syntax and Semantics of Independence Logic}\label{Subsection: Syntax and Semantics of Independence Logic}

To rigorously define the semantics of independence logic---an extension of first-order logic by the independence atom---we need to be more precise about our definitions. For the sake of some technical details later on, we consider team semantics in the context of many-sorted structures (see e.g.~\cite{sep-logic-many-sorted}).

\begin{definition}
    A (many-sorted relational) \emph{vocabulary} $\tau$ is a tuple $(\sorts_\tau,\rel_\tau,\arity_\tau,\sort_\tau)$ such that
    \begin{enumerate}
        \item $\rel_\tau$ is a set of relation symbols\footnote{We do not consider function symbols, as we do not need them later. They can be handled similarly to relation symbols.} and $\sorts_\tau\subseteq\N$,
        \item $\arity_\tau\colon\rel_\tau\to\N$ and $\sort_\tau\colon\rel_\tau\to\N^{<\omega}$ are functions with $\sort_\tau(R)\in\sorts_\tau^{\arity_\tau(R)}$ for $R\in\rel_\tau$, and
        \item if $n_i\in\N$, $i<k$, are such that $\sort_\tau(R)=(n_0,\dots,n_{k-1})$ for some $R\in \rel_\tau$, then $n_0,\dots,n_{k-1}\in\sorts_\tau$.
    \end{enumerate}
     We call $\arity_\tau(R)$ the arity of $R$ and $\sort_\tau(R)$ the sort of $R$.
    For $n\notin\sorts_\tau$, we say that a vocabulary $\tau'$ is the expansion of $\tau$ by the sort $n$ if
    \[
        \tau'=(\sorts_\tau\cup\{n\},\rel_\tau,\sort_\tau,\arity_\tau).
    \]

    A (many-sorted) \emph{$\tau$-structure} is a function $\A$ defined on the set $\rel_\tau\cup\sorts_\tau$ such that
    \begin{enumerate}
        \item $\A(n)$ is a nonempty set $A_n$ for $n\in\sorts_\tau$ and called the sort $n$ domain of $\A$, and
        \item $\A(R)\subseteq A_{n_0}\times\dots\times A_{n_{k-1}}$ for $R\in\rel_\tau$, where $\sort_\tau(R)=(n_0,\dots,n_{k-1})$.
    \end{enumerate}
    If $\tau'$ is an expansion of $\tau$ by sort $n$, we call a $\tau'$-structure $\B$ an \emph{expansion of $\A$ by the sort $n$} when $\B\restriction(\rel_\tau\cup\sorts_\tau)=\A$.
\end{definition}

We usually denote $\A(R)$ simply by $R^\A$ and $\A(n)$ by $A_n$. If $\A$ only has one sort, then we denote the domain of that sort by $A$ and call it the domain of $\A$. When there is no risk for confusion, we write $\arity$ and $\sort$ for $\arity_\tau$ and $\sort_\tau$.

For each sort $n\in\N$, we designate a set $\{v_i^n \mid i\in\N\}$ of variables of sort $n$, although for simplicity of notation we usually use symbols like $x$, $y$ and $z$ for variables and indicate the sort by writing $\sort(x)$ for the sort of $x$.

\begin{definition}[Syntax of Independence Logic]
    \label{Definition: Syntax of independence logic}
    The set of \emph{$\tau$-formulas} of independence logic is defined as follows.
    \begin{enumerate}
        \item First-order atomic and negated atomic formulas $u=v$, $\neg u=v$, $R(\vec{x})$ and $\neg R(\vec{x})$, where $R\in\rel_\tau$, $\vec{x}=(x_0,\dots,x_{\arity(R)-1})$ and $v$, $u$ and $x_i$ are variables with $\sort(u),\sort(v),\sort(x_i)\in\sorts_\tau$, $\sort(u)=\sort(v)$\footnote{We introduce identity for variables of the same sort only. This is not necessary but is sufficient here.} and $\sort(R)=(\sort(x_0),\dots,\sort(x_{\arity(R)-1}))$, are $\tau$-formulas.

        \item Independence atoms $\vec{y}\perp_{\vec{x}}\vec{z}$, where $\vec{x}=(x_0,\dots,x_{n-1})$, $\vec{y}=(y_0,\dots,y_{m-1})$ and $\vec{z}=(z_0,\dots,z_{l-1})$, and $x_i$, $y_j$ and $z_k$ are variables with $\sort(x_i),\sort(y_j),\sort(z_k)\in\sorts_\tau$, are $\tau$-formulas.

        \item If $\varphi$ and $\psi$ are $\tau$-formulas, then so are $\varphi\land\psi$ and $\varphi\lor\psi$.

        \item If $\varphi$ is a $\tau$-formula and $v$ is a variable with $\sort(v)\in\sorts_\tau$, then also $\forall v\varphi$ and $\exists v\varphi$ are $\tau$-formulas.
    \end{enumerate}
    We call \emph{dependence logic} the fragment of independence logic where only independence atoms of the form $\vec{y}\perp_{\vec{x}}\vec{y}$ are allowed.
    
    In addition to the usual syntax of independence logic, we introduce new quantifiers $\tilde{\forall}$ and $\tilde{\exists}$ which we will interpret as \emph{new sort quantifiers}. Similar quantifiers---although second order---were introduced by \jouko in~\cite{MR3205075}.
    
    \begin{enumerate}[resume]
        \item If $v$ is a variable such that $\sort(v)\notin\sorts_\tau$, $\tau'$ is the expansion of $\tau$ by the sort $\sort(v)$ and $\varphi$ is a $\tau'$-formula such that no variable, other than $v$, of sort $\sort(v)$ occurs free in $\varphi$, then $\tilde{\forall}v\varphi$ and $\tilde{\exists}v\varphi$ are $\tau$-formulas.
    \end{enumerate}
\end{definition}

The underlying idea of the new sort quantifiers $\tilde{\exists}$ and $\tilde{\forall}$ will become apparent in Definition~\ref{Definition: Semantics of independence logic} below.

\begin{definition}
    Let $\A$ be a $\tau$-structure and $D$ a set of variables. An \emph{assignment} $s$ of $\A$ with domain $D$ is a function $D\to\bigcup_{n\in\sorts_\tau}A_n$ such that $s(v)\in A_{\sort(v)}$ for all $v\in D$. If $s$ is an assignment of $\A$ with domain $D$, we write $s\colon D\to\A$. A \emph{team} $X$ of $\A$ with domain $D$ is a set of assignments of $\A$ with domain $D$. We denote by $\dom(X)$ the set $D$ and by $\rng(X)$ the set $\{s(v) \mid v\in D, s\in X\}$. If $X$ contains every assignment of $\A$, we call $X$ the full team of $\A$.
\end{definition}

For an assignment $s\colon D\to\A$, a variable $v$ (not necessarily in $D$) and $a\in A_{\sort(v)}$, we denote by $s(a/v)$ the assignment $D\cup\{v\}\to\A$ that maps $v$ to $a$ and $w$ to $s(w)$ for $w\in D\setminus\{v\}$. If $\vec{x}=(x_0,\dots,x_{n-1})$ is a tuple of variables, we denote by $s(\vec{x})$ the tuple $(s(x_0),\dots,s(x_{n-1}))$.

Given a team $X$ of $\A$, a variable $v$ and a function $F\colon X\to\Pow(A_{\sort(v)})\setminus\{\emptyset\}$, we denote by $X[F/v]$ the (``supplemented'') team $\{s(a/v) \mid s\in X, a\in F(s)\}$ and by $X[A_{\sort(v)}/v]$ the (``duplicated'') team $\{s(a/v) \mid s\in X, a\in A_{\sort(v)}\}$.

\begin{definition}[Semantics of Independence Logic]
    \label{Definition: Semantics of independence logic}
    Let $\tau$ be a (possibly many-sorted) vocabulary, $\A$ a $\tau$-structure, $X$ a team of $\A$ and $\varphi$ a $\tau$-formula. We then define the concept of the team $X$ satisfying the formula $\varphi$ in the structure $\A$, in symbols $\A\models_X\varphi$, as follows.\footnote{Strictly speaking, this definition is given in set theory for each quantifier rank separately. A uniform definition is possible if we put an upper bound on the cardinality of models considered.}
    \begin{enumerate}
        \item If $\varphi$ is a first-order atomic or negated atomic formula, then $\A\models_X\varphi$ if every assignment $s\in X$ satisfies $\varphi$ in $\A$ in the usual sense.
        \item If $\varphi = \vec{y}\perp_{\vec{x}}\vec{z}$, then $\A\models_X\varphi$ if for any $s,s'\in X$ with $s(\vec{x})=s'(\vec{x})$ there exists $s''\in X$ with $s''(\vec{x}\vec{y})=s(\vec{x}\vec{y})$ and $s''(\vec{z})=s'(\vec{z})$.
        \item If $\varphi=\psi\land\theta$, then $\A\models_X\varphi$ if $\A\models_X\psi$ and $\A\models_X\theta$.
        \item If $\varphi=\psi\lor\theta$, then $\A\models_X\varphi$ if $\A\models_Y\psi$ and $\A\models_Z\theta$ for some teams $Y$ and $Z$ such that $Y\cup Z = X$.
        \item If $\varphi=\forall v\psi$, then $\A\models_X\varphi$ if $\A\models_{X[A_{\sort(v)}/v]}\psi$.
        \item If $\varphi=\exists v\psi$, then $\A\models_X\varphi$ if $\A\models_{X[F/v]}\psi$ for some function $F\colon X\to\Pow(A_{\sort(v)})\setminus\{\emptyset\}$.
        \item If $\varphi=\tilde{\forall}v\psi$, then $\A\models_X\varphi$ if $\B\models_X\forall v\psi$ for all expansions $\B$ of $\A$ by the sort ${\sort(v)}$.
        \item If $\varphi=\tilde{\exists}v\psi$, then $\A\models_X\varphi$ if $\B\models_X\exists v\psi$ for some expansion $\B$ of $\A$ by the sort ${\sort(v)}$.
    \end{enumerate}
\end{definition}

If we restrict our attention to vocabularies and structures with just one sort, we get exactly the ordinary team semantics of independence logic.

When the underlying structure $\A$ is clear from the context or is irrelevant to the discussion (e.g. when the formula $\varphi$ does not contain any non-logical symbols or variables of multiple sorts), we simply write $X\models\varphi$ instead of $\A\models_X\varphi$.

\subsection{Axioms of Independence Logic}\label{subsec:axioms-of-independence-logic}

Although logical consequence in team semantics cannot be completely axiomatized (see the beginning of Section~\ref{subsec:relationships-between-the-properties}), it makes sense to isolate axioms that suffice for proving as many of the interesting logical consequences as possible. The rules we present here are, of course, then not intended to be complete in any sense; rather, they are just what we need in this paper. The general question of a more complete set of rules and axioms remains open.

\begin{definition}[Axioms of the Independence Atom,~\cite{MR3038039, MR3329287}]
    \label{axioms}
    The \textbf{axioms} of the independence atom are:
    \begin{enumerate}
        \item $\vec{y}\perp_{\vec{x}}\vec{y}$ entails $\vec{y}\perp_{\vec{x}}\vec{z}$. (Constancy Rule)
        \item $\vec{x}\perp_{\vec{x}}\vec{y}$. (Reflexivity Rule)
        \item   $\vec{z}\perp_{\vec{x}}\vec{y}$ entails $\vec{y}\perp_{\vec{x}}\vec{z}$. (Symmetry Rule)
        \item $\vec{y}{y'}\perp_{\vec{x}}\vec{z}{z'}$ entails $\vec{y}\perp_{\vec{x}}\vec{z}$. (Weakening Rule)
        \item If $\vec{z'}$ is a permutation of $\vec{z}$, $\vec{x'}$ is a permutation of $\vec{x}$, $\vec{y'}$ is a permutation of $\vec{y}$, then $\vec{y}\perp_{\vec{x}}\vec{z}$ entails $\vec{y'}\perp_{\vec{x'}}\vec{z'}$. (Permutation Rule)
        \item $\vec{z}\perp_{\vec{x}}\vec{y}$ entails $\vec{y}\vec{x}\perp_{\vec{x}}\vec{z}\vec{x}$. (Fixed Parameter Rule)
        \item $\vec{x}\perp_{\vec{z}}\vec{y}\wedge\vec{u}\perp_{\vec{z}\vec{x}}\vec{y}$ entails $\vec{u}\perp_{\vec{z}}\vec{y}$. (First Transitivity Rule)
        \item $\vec{y}\perp_{\vec{z}}\vec{y}\wedge\vec{z}\vec{x}\perp_{\vec{y}}\vec{u}$ entails $\vec{x}\perp_{\vec{z}}\vec{u}$. (Second Transitivity Rule)
        \item $\vec{x}\perp_{\vec{z}}\vec{y}\land\vec{x}\vec{y}\perp_{\vec{z}}\vec{u}$ entails $\vec{x}\perp_{\vec{z}}\vec{y}\vec{u}$. (Exchange Rule)
    \end{enumerate}
\end{definition}

The so-called Armstrong's Axioms for the dependence atom~\cite{MR0421121} follow from the above axioms:

\begin{definition}[Axioms of Dependence Atom]
    The \textbf{axioms} of the dependence atom are:
    \begin{enumerate}
        \item $\=(\vx\vy,\vx)$. (Reflexivity)
        \item $\=(\vx,\vy)\land\=(\vy,\vz)$ entails $\=(\vx,\vz)$. (Transitivity)
        \item $\=(\vx,\vy)$ entails $\=(\vx,\vx\vy)$. (Extensivity)
        \item If $\vec{x'}$ and $\vec{y'}$ are permutations of $\vx$ and $\vy$, respectively, then $\=(\vx,\vy)$ entails $\=(\vec{x'},\vec{y'})$. (Permutation)
    \end{enumerate}
\end{definition}
Armstrong's axioms are complete for dependence atoms, i.e. if $\Sigma$ is a set of dependence atoms and $\varphi$ is another dependence atom, then $\Sigma\models\varphi$ if and only if $\Sigma$ entails $\varphi$ by repeated applications of Armstrong's axioms.

The notion of a graphoid was introduced in~\cite{pearl1986graphoids} after the observation that certain axioms that hold true for conditional independence in probability theory are also satisfied by the vertex separation relation in a(n undirected) graph. A semigraphoid is a weakening of a graphoid, excluding one axiom. The axioms as we present them can be found in~\cite{pearl1988probabilistic}.

\begin{definition}
    The following are the semigraphoid \textbf{axioms}:
    \begin{enumerate}[label=(S\arabic*)]
        \item $\vec x\perp_{\vec z}\emptyset$. (Triviality)
        \item $\vec{x}\perp_{\vec{z}}\vec{y}$ entails $\vec{y}\perp_{\vec{z}}\vec{x}$. (Symmetry)
        \item $\vec x\perp_{\vec z}\vec y\vec w$ entails $\vec x\perp_{\vec z}\vec y$. (Decomposition)
        \item $\vec x\perp_{\vec z}\vec y\vec w$ entails $\vec x\perp_{\vec z\vec w}\vec y$. (Weak Union)
        \item $\vec{x}\perp_{\vec{z}}\vec{y} \land \vec{x}\perp_{\vec{z}\vec y}\vec{w}$ entails $\vec{x}\perp_{\vec{z}}\vec{y}\vec w$. (Contraction)
    \end{enumerate}
    In the original definition of a (semi)graphoid, $\vec x$, $\vec y$, $\vec z$ and $\vec w$ are sets instead of tuples, so we add the following (trivially valid) axioms to accommodate that:
    \begin{enumerate}[resume]
        \item If $\vec{x'}$, $\vec{y'}$ and $\vec{z'}$ are permutations of $\vec x$, $\vec y$ and $\vec{z}$, respectively, then $\vec{x}\perp_{\vec{z}}\vec{y}$ entails $\vec{x'}\perp_{\vec{z'}}\vec{y'}$. (Permutation)
        \item $\vec{x}\perp_{\vec{z}}\vec{y}$ entails $\vec{x}\vec x\perp_{\vec{z}\vec z}\vec{y}\vec y$. (Repetition)
    \end{enumerate}
\end{definition}

It is straightforward to show that the semigraphoid axioms are sound in team semantics (a fact which also follows from Proposition~\ref{Proposition (Hannula): Semigraphoid axioms hold very generally}). Next we show that the axioms of independence atom follow from the semigraphoid axioms, save the Reflexivity Rule. It remains open whether the Reflexivity Rule also follows from the semigraphoid axioms.

\begin{proposition}\label{Proposition: Axioms of independence atom follow from semigraphoid axioms}
    The axioms of the independence atom are provable from the semigraphoid axioms + the Reflexivity Rule.
\end{proposition}
\begin{proof}
    \begin{enumerate}
        \item Constancy Rule: Reflexivity rule gives $\vec x\vec y\perp_{\vec x\vec y}\vec z$. By a combination of symmetry, permutation and decomposition, from this we get $\vec y\perp_{\vec x\vec y}\vec z$. From $\vec y\perp_{\vec x}\vec y \land \vec y\perp_{\vec x\vec y}\vec z$ contraction gives $\vec y\perp_{\vec x}\vec y\vec z$, whence by decomposition we get $\vec y\perp_{\vec x}\vec z$.

        \item Reflexivity Rule: is an assumption.

        \item Symmetry Rule: This is the same as the symmetry axiom of semigraphoids.

        \item Weakening Rule: By the decomposition axiom, $\vec{y}{y'}\perp_{\vec{x}}\vec{z}{z'}$ entails $\vec{y}{y'}\perp_{\vec{x}}\vec{z}$, which by symmetry entails $\vec{z}\perp_{\vec{x}}\vec{y}{y'}$, which by decomposition entails $\vec{z}\perp_{\vec{x}}\vec{y}$, which by symmetry gives $\vec{y}\perp_{\vec{x}}\vec{z}$.

        \item Permutation Rule: This is the same as permutation of semigraphoids.

        \item Fixed Parameter Rule: From the Reflexivity Rule and symmetry, we get $\vec y\perp_{\vec x}\vec x$. From $\vec y\perp_{\vec x}\vec z$ we get $\vec y\perp_{\vec x\vec x}\vec z$ by repetition. From $\vec y\perp_{\vec x}\vec x \land \vec y\perp_{\vec x\vec x}\vec z$ contraction gives $\vec y\perp_{\vec x}\vec x\vec z$. By symmetry and repetition, we have $\vec x\vec z\perp_{\vec x\vec x}\vec y$, and again reflexivity + symmetry gives $\vec x\vec z\perp_{\vec x}\vec x$. From $\vec x\vec z\perp_{\vec x}\vec x \land \vec x\vec z\perp_{\vec x\vec x}\vec y$ contraction again gives $\vec x\vec z\perp_{\vec x}\vec x\vec y$. By symmetry + permutation this yields $\vec y\vec x\perp_{\vec x}\vec z\vec x$ as desired.

        \item First Transitivity Rule: By symmetry, from $\vec{x}\perp_{\vec{z}}\vec{y}$ we get $\vec{y}\perp_{\vec{z}}\vec{x}$ and from $\vec{u}\perp_{\vec{z}\vec{x}}\vec{y}$ we get $\vec{y}\perp_{\vec{z}\vec{x}}\vec{u}$. Applying contraction to $\vec{y}\perp_{\vec{z}}\vec{x}$ and $\vec{y}\perp_{\vec{z}\vec{x}}\vec{u}$, we get $\vec{y}\perp_{\vec{z}}\vec{x}\vec{u}$, from which the weakening rule that we already proved gives $\vec{y}\perp_{\vec{z}}\vec{u}$. Then by symmetry we get $\vec{u}\perp_{\vec{z}}\vec{y}$.

        \item Second Transitivity Rule: From $\vec{z}\vec{x}\perp_{\vec{y}}\vec{u}$, symmetry, permutation and weak union give $\vec{x}\perp_{\vec{z}\vec{y}}\vec{u}$. From $\vec y\perp_{\vec z}\vec y$ constancy rule + symmetry gives $\vec x\perp_{\vec z}\vec y$. Then contraction gives $\vec x\perp_{\vec z}\vec y\vec u$, whence by decomposition we obtain $\vec x\perp_{\vec z}\vec u$.

        \item Exchange Rule: From $\vec{x}\vec{y}\perp_{\vec{z}}\vec{u}$, symmetry + weak union gives $\vec{x}\perp_{\vec{z}\vec{y}}\vec{u}$. Then from $\vec{x}\perp_{\vec{z}}\vec{y}$ and $\vec{x}\perp_{\vec{z}\vec{y}}\vec{u}$, contraction gives $\vec{x}\perp_{\vec{z}}\vec{y}\vec{u}$.
    \end{enumerate}
\end{proof}

So it turns out that the semigraphoid axioms, with the Reflexivity Rule added, are sufficient to prove all the others. This will be useful in Section~\ref{sec:probabilistic-teams}. However, we will use all of the above axioms in the sequel.

Next we add rules for conjunction and existential quantifier, as we shall be working mainly with an existential-conjunctive fragment of independence logic.

\begin{definition}[Quantifiers and Connectives,~\cite{MR3079984,MR3270139}]
    \label{Definition: Rules for logical operations}\quad
    \begin{enumerate}
        \item The following is the elimination rule for existential quantifier:
        \begin{quote}
            If $\Sigma$ is a set of formulas, $\Sigma\cup\{\varphi\}$ entails $\psi$ and $x$ does not occur free in $\psi$ or in any $\theta\in\Sigma$, then $\Sigma\cup\{\exists x\varphi\}$ entails $\psi$.
        \end{quote}
        \item The following is the introduction rule for existential quantifier:
        \begin{quote}
            If $y$ does not occur in the scope of $Qx$ in $\varphi$ for any $Q\in\{\exists,\forall,\tilde{\exists},\tilde{\forall}\}$, then $\varphi(y/x)$ (i.e. the formula one obtains by replacing every free occurrence of $x$ in $\varphi$ by $y$) entails $\exists x\varphi$.
        \end{quote}
        \item The following is the elimination rule for conjunction:
        \begin{quote}
            $\varphi\land\psi$ entails both $\varphi$ and $\psi$.
        \end{quote}
        \item The following is the introduction rule for conjunction:
        \begin{quote}
            $\{\varphi,\psi\}$ entails $\varphi\land\psi$.
        \end{quote}
       \item The following is the rule for dependence introduction:
        \begin{quote}
            $\exists x\varphi$ entails $\exists x (\=(\vec{z},x)\land\varphi)$ whenever $\varphi$ is a formula of dependence logic, where $\vec{z}$ lists the free variables of $\exists x\varphi$.
        \end{quote}
        \item The following is the first introduction rule for $\tilde\exists$:
        \begin{quote}
            If no variable with sort $\sort(x)$ occurs in any formula of $\Sigma$ and $\Sigma$ entails $\exists x\varphi$, then $\Sigma$ entails $\tilde\exists x\varphi$.
        \end{quote}
        \item The following is the second introduction rule for $\tilde\exists$:
        \begin{quote}
            If no variable with sort $\sort(x)$ occurs in any formula of $\Sigma$ and $\Sigma\cup\{\exists x\varphi\}$ entails $\exists y\psi$, where $\sort(y) = \sort(x)$, then $\Sigma\cup\{\tilde\exists x\varphi\}$ entails $\tilde\exists y\psi$.
        \end{quote}
    \end{enumerate}
\end{definition}

For the new sort existential quantifier  $\tilde\exists$, we only give the above rather immediate axioms. These rules could possibly be strengthened by allowing variables of sort $\sort(x)$ occur in the scope of $\tilde\exists y$ for $\sort(y) = \sort(x)$ in formulas of $\Sigma$.
It may also be interesting to look for more axioms for the sort quantifiers. For example, $\tilde\exists x\exists y \neg x=y$ for $\sort(x)=\sort(y)$ is a natural valid sentence (even though the sentence $\exists x\exists y \neg x=y$ is not valid) but apparently not derivable from our current axioms.

\begin{proposition}[Soundness Theorem]
    If $\varphi$ entails $\psi$ by repeated applications of the rules of Definitions~\ref{axioms}--\ref{Definition: Rules for logical operations}, then $\varphi\models\psi$ in team semantics.
\end{proposition}
\begin{proof}
    We show that the rules for $\tilde\exists$ are sound. For the second introduction rule, suppose that $\Sigma\cup\{\exists x\varphi\}\models\exists y\psi$, where $\sort(y) = \sort(x)$. Then suppose that $\A\models_X\Sigma\cup\{\tilde\exists x\varphi\}$. Then $\A$ has an expansion $\A^*$ by the sort $\sort(x)$ with $\A^*\models_X\Sigma\cup\{\exists x\varphi\}$. Thus $\A^*\models_X\exists y\psi$, whence $\A\models_X\tilde\exists y\psi$. The first introduction rule is the same but without the assumption $\tilde\exists x\varphi$.
\end{proof}

If $\varphi$ entails $\psi$ by repeated applications of the above rules, we write $\varphi\entails\psi$.

\section{Logical Properties of Teams}\label{sec:properties-of-teams}

Quantum physics provides a rich source of highly non-trivial dependence and independence concepts. Some of the most fundamental questions of quantum physics concern independence of outcomes of experiments. \Samson presented in~\cite{MR3038040} a relational (possibilistic) approach to model these dependence and independence phenomena. His framework very naturally transforms into a team-semantic adaptation which we will carry out now.

\subsection{Empirical and Hidden-Variable Teams}\label{sec:empirical-and-hidden-variable-teams}

As discussed in Section~\ref{dail}, we consider teams with designated variables for measurements and separate variables for outcomes. An important role in models of quantum physics is played by the so-called \emph{hidden variables}, variables 
which are not directly observable, but
which play a role in determining the outcomes of measurements, explaining indeterministic or non-local behaviour. The following terminology and notation is helpful in dealing with teams arising in this way in relation to  quantum physics.

We use a division of variables into three sorts, defined below. A priori there is no difference between the variables. 
This division into three sorts is simply helpful in guiding our intuitions.
Our purely abstract results about teams based on these variables help us organize quantum-theoretic concepts. However, it is worth noting that e.g.~the word ``measurement'' has a meaning that corresponds to a physical event and the assumptions we make in the form of the properties of teams presented in Section~\ref{subsec:properties-of-empirical-teams}, of course reflect the properties---observed or postulated---of these physical events.

\begin{definition}
    Fix
    \begin{itemize}
        \item a set $\Vm = \{x_0,\dots,x_{n-1}\}$ of measurement variables,
        \item a corresponding set $\Vo = \{y_0,\dots,y_{n-1}\}$ of outcome variables, and
        \item a set $\Vh = \{z_0,\dots,z_{l-1}\}$ of hidden variables.
    \end{itemize}
    We say that a team $X$ is an \emph{empirical team} if $\dom(X)=\Vm\cup\Vo$. We say that a team $X$ is a \emph{hidden-variable team} if $\dom(X)=\Vm\cup\Vo\cup\Vh$.
\end{definition}

Throughout the paper, we will denote by $n$ the number of measurement and outcome variables and by $l$ the number of hidden variables.

Definition~\ref{5.2} below makes the connection between our concept of an empirical team and the mathematical model predicting what the possible outcomes of experiments could be, namely the theory of operators of complex Hilbert spaces, explicit.

We will pay special attention to \emph{definability} of properties of teams. In other words, if $P$ is a property of teams, especially of empirical or hidden-variable teams, we ask whether there is a formula $\varphi$ of independence logic with the free variables $\Vm\cup\Vo$ (or $\Vm\cup\Vo\cup\Vh$) which is satisfied in the sense of team semantics exactly by those teams that have the property $P$.

A hidden-variable team is a team of the form

\begin{center}
    $Y=$
    \begin{tabular}{|cccccccc|}
        $x_0\phantom{{}^{-1}}$    & $y_0\phantom{{}^{-1}}$    & $\ldots$ & $x_{n-1}$       & $y_{n-1}$       & $z_0\phantom{{}^{-1}}$
        & $\ldots$
        & $z_{l-1}\phantom{{}^{1}}$
        \\
        \hline
        $a^0_0\phantom{{}^{-1}}$  & $b^0_0\phantom{{}^{-1}}$  & $\ldots$ & $a^0_{n-1}$     & $b^0_{n-1}$
        & $\gamma^0_0\phantom{{}^{-1}}$
        & $\ldots$
        & $\gamma^0_{l-1}\phantom{{}^{1}}$
        \\
        $a^2_0\phantom{{}^{-1}}$  & $b^1_0\phantom{{}^{-1}}$  & $\ldots$ & $a^2_{n-1}$     & $b^1_{n-1}$
        & $\gamma^1_0\phantom{{}^{-1}}$
        & $\ldots$
        & $\gamma^2_{l-1}\phantom{{}^{1}}$
        \\
        $\vdots\phantom{{}^{-1}}$ & $\vdots\phantom{{}^{-1}}$ & $\ddots$ & $\vdots$        & $\vdots$        & $\vdots\phantom{{}^{-1}}$
        & $\ddots$
        & $\vdots\phantom{{}^{1}}$
        \\
        $a^{m-1}_0$               & $b^{m-1}_0$               & $\ldots$ & $a^{m-1}_{n-1}$ & $b^{m-1}_{n-1}$ & $\gamma^{m-1}_0$
        & $\ldots$
        & $\gamma^{m-1}_{l-1}$
        \\
    \end{tabular}
\end{center}
where the $\gamma^i_j$ indicate values which we cannot observe directly. 
A typical hidden variable is some kind of ``state'' of the system. 

Every team has a background model from which the values of assignments come. In a many-sorted context the background model has one universe for each sort. The universes may intersect. We assume a universe also for the hidden-variable sort.

\begin{definition}[\cite{MR3038040}]\label{equivalent}
    A hidden-variable team $Y$ \emph{realizes} an empirical team $X$ if
    \[
        s\in X \iff \exists s'\in Y \bigwedge_{i<n}(s'(x_i)=s(x_i) \land s'(y_i)=s(y_i)).
    \]
    Two hidden-variable teams are said to be \emph{(empirically) equivalent} if they realize the same empirical team.
\end{definition}

The property of being the realization of a hidden-variable team is definable in independence logic. It can be defined simply by the existential quantifier: If $\varphi(\vec{x},\vec{y},\vec{z})$ is a formula of independence logic, and thereby defines a property of teams, then $\tilde{\exists} z_0\exists z_1\dots\exists z_{l-1}\varphi$ defines the class of empirical teams that are realized by some hidden-variable teams satisfying $\varphi$. The ``hidden'' character of the hidden variables is built into the semantics of the sort quantifier.

Realization of an empirical team by a hidden-variable team involves a kind of projection where one projects away the hidden variables. Hidden-variable teams are divided into equivalence classes according to whether they project into the same empirical team or not. This phenomenon can of course be thought of more generally: for any set $V$ of variables and $V'\subseteq V$, one can define a projection mapping $\Pr_{V'}$ such that if $X$ is a team with domain $V$, then $\Pr_{V'}(X) = \{s\restriction V' \mid s\in X \}$.

Next we use the resources of independence logic with its team semantics to express properties of empirical and hidden-variable teams. The possible benefits of expressing such properties in the formal language of independence logic are two-fold. First, the quantum-theoretic concepts may suggest interesting new facts about independence logic in general, applicable perhaps also in other fields. Second, concepts, proofs and constructions of independence logic may shed new light on connections between concepts in quantum physics, and may focus attention on what is particular to quantum physics, and what are merely general logical facts about independence concepts.

\subsection{Properties of Empirical Teams}\label{subsec:properties-of-empirical-teams}

We observe that the definitions of the simpler properties of empirical teams treated by \samson in~\cite{MR3038040} can be expressed by formulas of independence logic, in fact a conjunction of independence atoms. For the original definitions, we refer to~\cite{Dickson_1998, 231dee39-4a55-385c-aad2-2b0e67edd310, Shimony_1993, MR3790629}.

As discussed in Section {\ref{dail}}, a team is said to support \emph{weak determinism} if each outcome is determined by the combination of all the measurement variables:

\begin{definition}[Weak Determinism]
    An empirical team $X$ supports \emph{weak determinism} if it satisfies the formula
    \[
        \bigwedge_{i<n} \=(\vec{x},y_i). \tag{WD}
    \]
\end{definition}

Thus weak determinism is expressed simply with a conjunction of dependence atoms. In fact, the meaning of the dependence atom $\=(x,y)$ is that $x$ completely determines $y$. Therefore saying that teams supporting~\eqref{eq: Weak Determinism} support weak determinism is appropriate. The only difference to the ordinary dependence atom is that in~\eqref{eq: Weak Determinism} we separate the variables into the measurements $x_i$ and the outcomes $y_i$.

A team is said to support \emph{strong determinism} if the outcome variable $y_i$ of any measurement is completely determined by the measurement variable $x_i$.

\begin{definition}[Strong Determinism]
    An empirical team $X$ supports \emph{strong determinism} if it satisfies the formula
    \[
        \bigwedge_{i<n} \=(x_i,y_i). \tag{SD}
    \]
\end{definition}

We now come to the important \emph{no-signalling} condition. The motivation for this comes from the physical scenario with which we started, in which the parties $i \in n$ are spacelike separated from each other.
This means that there can be no information flowing between the measurements performed by each party; in particular, which measurement was performed at party $i$ cannot influence what the possible outcomes of a given measurement are at another party $j$. More generally, the possible outcomes of given measurements at a set of parties $I \subseteq n$ cannot be influenced by which measurements are performed at the remaining parties $n \setminus I$.
Crucially, although quantum mechanics is non-local, it does satisfy no-signalling, and hence is consistent with relativity theory.

This condition is formalized as follows.
Suppose the team $X$ has two possible measurement-outcome combinations $s$ and $s'$ with inputs $x_i$, $i\in I$, the same. So now $s(\{y_i \mid i\in I\})$ is a possible outcome of the measurements $\{x_i \mid i\in I\}$ in view of $X$. We demand that $s(\{y_i \mid i\in I\})$ is also a possible outcome if the inputs $s(x_j)$, $j\notin I$, of the other experiments are changed to $s'(x_j)$.

\begin{definition}[No-Signalling]
    An empirical team $X$ supports \emph{no-signalling} if it satisfies the formula
    \begin{equation}
        \label{eq: No-Signalling}
        \bigwedge_{I\subseteq n}\{x_i \mid i\notin I\}\perp_{\{x_i \mid i\in I\}}\{y_i \mid i\in I\}. \tag{NS}
    \end{equation}
\end{definition}

In \cite{MR3038040}, a weaker version of no-signalling is presented where the subsets $I$ are singletons, so the corresponding formula would be $\bigwedge_{i<n}\{ x_j \mid j\neq i \}\perp_{x_i}y_i$.

In principle, supporting no-signalling means just satisfying a conjunction of independence atoms. But the atoms are of a particular form because of our division of variables into different sorts. The atom $\{x_i \mid i\notin I\}\perp_{\{x_i \mid i\in I\}}\{y_i \mid i\in I\}$ says that the outcomes $y_i$, $i\in I$, are meant to be related to the measurements $x_i$, $i\in I$, and be totally independent of the measurements $x_j$, $j\notin I$.

\subsection{Properties of Hidden-Variable Teams}\label{subsec:properties-of-hidden-variable-teams}

For hidden-variable teams, the hidden variables are added in the definition of determinism as extra variables that determine the outcomes of the system.

\begin{definition}[Weak Determinism]
    A hidden-variable team $X$ supports \emph{weak determinism} if it satisfies the formula
    \begin{equation}
        \label{eq: Weak Determinism}
        \bigwedge_{i<n} \=(\vec{x}\vec{z},y_i). \tag{WD}
    \end{equation}
\end{definition}

\begin{definition}[Strong Determinism]
    A hidden-variable team $X$ supports \emph{strong determinism} if it satisfies the formula
    \begin{equation}
        \label{eq: Strong Determinism}
        \bigwedge_{i<n} \=(x_i\vec{z},y_i). \tag{SD}
    \end{equation}
\end{definition}

A team $X$ is said to support \emph{single-valuedness} if each hidden variable $z_k$ can only take one value.

\begin{definition}[Single-Valuedness]
    A hidden-variable team $X$ supports \emph{single-valuedness} if it satisfies the formula
    \begin{equation}
        \label{eq: Single-Valuedness}
        \=(\vec{z}) \tag{SV}
    \end{equation}
\end{definition}
The formula $\=(\vz)$ is a so-called \emph{constancy atom}~\cite{MR2480819}, a degenerate form of the dependence atom $\=(\vx,\vy)$, where $\vx$ is the empty tuple.

A team $X$ is said to support \emph{$\vz$-independence} if the following holds: Suppose the team $X$ has two measurement-outcome combinations $s$ and $s'$. Now the hidden variables $\vec{z}$ have some value $s(\vec{z})$ in the combination $s$. We demand that $s(\vec{z})$ should occur as the value of the hidden variable also if the inputs $s(\vec{x})$ are changed to $s'(\vec{x})$.

\begin{definition}[$\vz$-Independence]
    A hidden-variable team $X$ supports \emph{$\vz$-inde\-pend\-ence} if it satisfies the formula
    \begin{equation}
        \label{eq: z-Independence}
        \vec{z}\perp\vec{x}. \tag{$\vz$\hspace{1pt}I}
    \end{equation}
\end{definition}

\emph{Parameter independence} is the hidden-variable version of no-signalling. A team $X$ is said to support parameter-independence if the following holds: Suppose the team $X$ has two measurement-outcome combinations $s$ and $s'$ with the same input data about ${x_i}$, $i\in I$, and the same hidden variables $\vec{z}$, i.e. $s(\{x_i \mid i\in I\})=s'(\{x_i\mid i\in I\})$ and $s(\vec{z})=s'(\vec{z})$. We demand that the outcome data $s(\{y_i \mid i\in I\})$ should occur as a possible outcome also if the inputs $s(\{x_j \mid j\notin I\})$ are changed to $s'(\{x_j \mid j\notin I\})$.

\begin{definition}[Parameter Independence]
    A hidden-variable team $X$ supports \emph{parameter independence} if it satisfies the formula
    \begin{equation}
        \label{eq: Parameter Independence}
        \bigwedge_{I\subseteq n}\{x_i \mid i\notin I\}\perp_{\{x_i \mid i\in I\}\vec z}\{y_i \mid i\in I\}. \tag{PI}
    \end{equation}
\end{definition}

Note that as with no-signalling, the version of parameter independence presented in \cite{MR3038040} would correspond to the formula $\bigwedge_{i<n}\{ x_j \mid j\neq i \}\perp_{x_i\vec{z}}y_i$.

A team $X$ is said to support \emph{outcome-independence} if the following holds: Suppose the team $X$ has two measurement-outcome combinations $s$ and $s'$ with the same total input data $\vec{x}$ and the same hidden variables $\vec{z}$, i.e. $s(\vec{x})=s'(\vec{x})$ and $s(\vec{z})=s'(\vec{z})$. We demand that outcome $s(y_i)$ should occur as an outcome also if the outcomes $s(\{y_j \mid j\ne i\})$ are changed to $s'(\{y_j \mid j\ne i\})$. In other words, the variables $y_i$, $i<n$, are mutually independent whenever $\vec{x}\vec{z}$ is fixed.

\begin{definition}[Outcome Independence]
    A hidden-variable team $X$ supports \emph{outcome independence} if it satisfies the formula
    \begin{equation}
        \label{eq: Outcome Independence}
        \bigwedge_{i<n} y_i\perp_{\vec{x}\vec{z}}\{y_j \mid j\neq i\}. \tag{OI}
    \end{equation}
\end{definition}

All the previous examples were, from the point of view of independence logic, atoms or conjunctions of atoms of the same kind with a certain organization of the variables. We shall now consider a property which is slightly more complicated.

This is the crucial notion of  \emph{locality}, which expresses the idea that the possible outcomes  of a party can only depend on the input to that party, together with the values of the hidden variables, and not on the outcomes of any other party.
This strengthens the no-signalling condition, which only requires independence from the \emph{inputs} of the other parties. Whereas quantum mechanics satisfies no-signalling, it violates locality---hence allowing for non-local correlations of outcomes.
This condition is formalized in team semantics as follows:

\begin{definition}[Locality]
    A hidden-variable team $X$ satisfies \emph{locality} if
    \begin{gather*}
        \forall s_0,\dots,s_{n-1}\in X \left[ \exists s\in X\bigwedge_{i<n} s(x_i \vec{z}) = s_i(x_i \vec{z}) \right. \\
        \left. \implies \exists s'\in X\bigwedge_{i<n}s'(x_i y_i \vec{z}) = s_i(x_i y_i \vec{z}) \right].
    \end{gather*}
\end{definition}

The definition of locality is not \emph{per se} an expression of independence logic. However, in Lemma~\ref{Lemma: Locality is equivalent to parameter independence + outcome independence} below we prove that locality can be defined, after all, by a conjunction of independence atoms.

\subsection{Relationships between the Properties}\label{subsec:relationships-between-the-properties}

We present several logical consequences of independence logic and demonstrate how they can be interpreted in the context of empirical and hidden-variable teams. In many cases we can derive the logical consequence relation from the axioms of Definition~\ref{axioms}. Semantic proofs are due to~\cite{MR3038040},

It should be noted that logical consequence in independence logic is in principle a highly complex concept. For example, it cannot be axiomatized because the set of G\"odel numbers of valid sentences (of even dependence logic\footnote{This observation is essentially due to A. Ehrenfeucht, see~\cite{MR0143691}.}) is non-arithmetical. Even the implication problem for the independence atoms is undecidable~\cite{MR1358026}, while for dependence atoms it is decidable~\cite{MR0421121}.
Logical implication between finite conjunctions of independence atoms is, however, recursively axiomatizable, as it can be reduced to logical consequence in first-order logic by introducing a new predicate symbol.

Because of the complexity of logical consequence, it is important to accumulate good examples. We claim that the below examples arising from quantum mechanics are illustrative examples and may guide us in finding a more systematic approach.

\begin{lemma}
    $\=(\vec{x}\vec{z},\vec{y})\entails \bigwedge_{i<n} y_i\perp_{\vec{x}\vec{z}}\{y_j \mid j\neq i\}$.
\end{lemma}

In words, if a hidden-variable team supports weak determinism, then it supports outcome independence.

\begin{proof}
    $\=(\vec{x}\vec{z},\vec{y})$ means $\vec{y}\perp_{\vec{x}\vec{z}}\vec{y}$. Given any $i<n$, one obtains $y_i\perp_{\vec{x}\vec{z}}\{y_j \mid j\neq i\}$ from $\vec{y}\perp_{\vec{x}\vec{z}}\vec{y}$ by a single application of the Weakening Rule of independence atoms.
\end{proof}

\begin{lemma}
    $\bigwedge_{i<n}\=(x_i\vec z,y_i)\entails\bigwedge_{I\subseteq n}\{x_i \mid i\notin I\}\perp_{\{x_i \mid i\in I\}\vec z}\{y_i \mid i\in I\}$.
\end{lemma}

In words, if a hidden-variable team supports strong determinism, then it supports parameter independence

\begin{proof}
    Fix $I\subseteq n$. Using Armstrong's axioms, one can obtain
    \[
        \{\=(x_i\vec{z},y_i) \mid i\in I\}\entails\=(\{x_i \mid i\in I\}\vz,\{y_i \mid i\in I\}).
    \]
    Note that $\=(\{x_i \mid i\in I\}\vz,\{y_i \mid i\in I\})$ means $\{y_i \mid i\in I\}\perp_{\{x_i \mid i\in I\}\vz}\{y_i \mid i\in I\}$. Now the Constancy Rule of independence atoms gives $\{y_i \mid i\in I\}\perp_{\{x_i \mid i\in I\}\vz}\vec{w}$ for any variable tuple $\vec{w}$, in particular when $\vec{w} = \{ x_i \mid i\notin I \}$. Finally, we obtain $\{x_i \mid i\notin I\}\perp_{\{x_i \mid i\in I\}\vec z}\{y_i \mid i\in I\}$ by using the Symmetry Rule.
\end{proof}

\begin{lemma}
    $\left( \bigwedge_{I\subseteq n}\{ x_i \mid i\notin I \}\!\perp_{\{x_i \mid i\in I\}\vec{z}}\!\{y_i \mid i\in I\} \right) \land \=(\vec{x}\vec{z},\vec{y})\entails\bigwedge_{i<n}\!\=(x_i\vec{z}, y_i)$.
\end{lemma}

In words, if a hidden-variable team supports parameter independence and weak determinism, then it supports strong determinism.

\begin{proof}
    Fix $i<n$. $\=(\vec{x}\vec{z},\vec{y})$ means $\vec{y}\perp_{\vec{x}\vec{z}}\vec{y}$, from which we get $y_i\perp_{\vec{x}\vec{z}}y_i$ using the Weakening Rule. From parameter independence we obtain $\{ x_j \mid j\neq i \}\perp_{x_i\vec{z}}y_i$ by choosing the conjunct with $I = n\setminus\{i\}$. Then we have
    \[
        \{ x_j \mid j\neq i \}\perp_{x_i\vec{z}}y_i \land y_i\perp_{\vec{x}\vec{z}}y_i.
    \]
    Finally, the First Transitivity Rule yields $y_i\perp_{x_i\vec{z}}y_i$, which means $\=(x_i\vec{z}, y_i)$.
\end{proof}

\begin{lemma}
    \label{Lemma: Locality is equivalent to parameter independence + outcome independence}
    Locality is equivalent to the formula
    \[
        \left(\bigwedge_{I\subseteq n} \{ x_i \mid i\notin I \}\perp_{\{x_i \mid i\in I\}\vec{z}}\{y_i \mid i\in I\}\right) \land \left(\bigwedge_{i<n} y_i\perp_{\vec{x}\vec{z}}\{y_j \mid j\neq i\} \right).
    \]
\end{lemma}

In words, a hidden-variable team $X$ supports locality if and only if it supports both parameter independence and outcome independence.

\begin{proof}

    It is essentially proved in~\cite{MR3038040} that the weaker version of parameter independence
    \[
        \bigwedge_{i<n}\{x_j \mid j\neq i\}\perp_{x_i\vz}y_i
    \]
    together with outcome independence is equivalent to locality. What is left is to show that the stronger version of parameter independence still follows from locality.
    So suppose that $X$ supports locality and fix $I\subseteq n$. We show that
    \[
       X\models\{ x_i \mid i\notin I \}\perp_{\{x_i \mid i\in I\}\vec{z}}\{y_i \mid i\in I\}.
    \]
    Let $s,s'\in X$ be such that $s(x_i\vec{z})=s'(x_i\vec{z})$ for all $i\in I$. We wish to find $s''\in X$ with $s''(\vx\vz)=s(\vx\vz)$ and $s''(\{y_i \mid i\in I\})=s'(\{y_i \mid i\in I\})$. Let $s_i = s'$ for $i\in I$ and $s_i = s$ for $i\notin I$. Now $s$ is such that for $i\in I$, $s(x_i\vec{z}) = s'(x_i\vec{z}) = s_i(x_i\vec{z})$, but also for $i\notin I$ we have $s(x_i\vec{z}) = s_i(x_i\vec{z})$ (as $s = s_i$). Hence we have $s(x_i\vec{z})=s_i(x_i\vec{z})$ for all $i<n$, so by locality there exists $s''\in X$ with $s''(x_iy_i\vec{z}) = s_i(x_iy_i\vec{z})$ for all $i<n$. But then $s''(x_i\vz) = s_i(x_i\vec{z}) = s(x_i\vec{z})$ for all $i<n$ and $s''(y_i) = s_i(y_i) = s'(y_i)$ for $i\in I$. Thus $s''$ is as desired.
\end{proof}

Next we indicate connections between properties of empirical teams and properties of hidden-variable teams, again following~\cite{MR3038040}.

\begin{proposition}\label{317}
    The sentence $\tilde\exists z_0\exists z_1\dots\exists z_{l-1} \=(\vec{z})$ is valid. More generally, if the variables $\vec{z}$ do not occur in $\varphi$, then $\varphi\entails\tilde\exists z_0\exists z_1\dots\exists z_{l-1} (\=(\vec{z})\wedge\varphi)$.
\end{proposition}

In words, every empirical team is realized by a hidden-variable team supporting single-valuedness.

\begin{proof}
    By the Reflexivity Rule, we have $\vec{z}\perp_{\vec{z}}\vec{z}$. Using introduction of existential quantifier $l$ times, we obtain $\exists z_0\dots\exists z_{l-1}\ \vec{z}\perp_{\vec{z}}\vec{z}$. Using elimination of existential quantifier and introduction of dependence $l$ times we obtain
    \[
        \exists z_0\dots\exists z_{l-1}\left(\bigwedge_{k<l}\=(z_k)\land\vec{z}\perp_{\vec{z}}\vec{z}\right).
    \]
    As, from Armstrong's axioms, one can infer $\bigwedge_{k<l}\=(z_k)\entails\=(\vec{z})$, we then easily obtain $\exists\vec{z}\=(\vec{z})$. Then, assuming $\varphi$, by using elimination and introduction of existential quantifier $l$ times, we obtain $\exists\vec{z}(\=(\vec{z})\land\varphi)$. Finally, the first introduction rule of $\tilde\exists$ gives $\tilde\exists z_0\exists z_1\dots\exists z_{l-1} (\=(\vec{z})\wedge\varphi)$.
\end{proof}

\begin{proposition}
    \label{Proposition: An empirical team supports no-signalling iff it is realized by a hidden-variable team supporting z-independence and parameter independence}
    Let
    \[
        \varphi = \bigwedge_{I\subseteq n}\{x_i \mid i\notin I\}\perp_{\{x_i\mid i\in I\}}\{y_i \mid i\in I\}
    \]
    and
    \[
        \psi = \tilde{\exists}z_0\exists z_1\dots\exists z_{l-1} \left( \vec{z}\perp\vec{x} \land \bigwedge_{I\subseteq n} \{x_i \mid i\notin I\}\perp_{\{x_i\mid i\in I\}\vec{z}}\{y_i\mid i\in I\} \right).
    \]
    Then $\varphi \dashv\vdash \psi$.
\end{proposition}

In words, an empirical team supports no-signalling if and only if it can be realized by a hidden-variable team supporting $\vz$-independence and parameter independence.

\begin{proof}
    We first show that $\varphi\entails\psi$. First of all, assume $\=(\vz)$, which is harmless in view of Proposition~\ref{317}. Note that $\=(\vz)$ means $\vec z\perp\vec z$. By the Constancy Rule, $\vec z\perp\vec z$ entails $\vec z\perp\vec x$. Then fix $I\subseteq n$. Note that $\=(\vz)$ also means $\=(\emptyset,\vz)$. Now from Armstrong's axioms one can obtain $\=(\vec w,\vz)$ for any variable tuple $\vec w$, in particular when $\vec w = \{x_i \mid i\in I\}\cup\{y_i \mid i\in I\}$. Now $\=(\{x_i \mid i\in I\}\{y_i \mid i\in I\},\vec z)$ means $\vec z\perp_{\{x_i \mid i\in I\}\{y_i \mid i\in I\}}\vec z$. By the Constancy Rule, $\vec z\perp_{\{x_i \mid i\in I\}\{y_i \mid i\in I\}}\vec z$ entails $\{x_i \mid i\notin I\}\perp_{\{x_i \mid i\in I\}\{y_i \mid i\in I\}}\vec z$. Then by Contraction, $\{x_i \mid i\notin I\}\perp_{\{x_i \mid i\in I\}}\{y_i \mid i\in I\}$ and $\{x_i \mid i\notin I\}\perp_{\{x_i \mid i\in I\}\{y_i \mid i\in I\}}\vec z$ together entail $\{ x_i \mid i\notin I \}\perp_{\{x_i \mid i\in I\}}\{y_i \mid i\in I\}\vec{z}$, which by Weak Union entails $\{ x_i \mid i\notin I \}\perp_{\{x_i \mid i\in I\}\vec{z}}\{y_i \mid i\in I\}$.
    
    Hence from the assumptions $\=(\vec z)$ and $\{x_i \mid i\notin I\}\perp_{\{x_i \mid i\in I\}}\{y_i \mid i\in I\}$, we can deduce $\vec z\perp\vec x \land \{ x_i \mid i\notin I \}\perp_{\{x_i \mid i\in I\}\vec{z}}\{y_i \mid i\in I\}$. Introducing conjunctions and existential quantifiers, we obtain that $\=(\vz)$ and $\varphi$ entail
    \[
        \exists z_0\dots\exists z_{l-1} \left( \vec{z}\perp\vec{x} \land \bigwedge_{I\subseteq n} \{x_i \mid i\notin I\}\perp_{\{x_i\mid i\in I\}\vec{z}}\{y_i\mid i\in I\} \right),
    \]
    and hence by eliminating the existential quantifiers of the formula $\exists z_0\dots\exists z_{l-1}\=(\vec z)$, we obtain that $\exists z_0\dots\exists z_{l-1}\=(\vec z)$ and $\varphi$ together entail
    \[
        \exists z_0\dots\exists z_{l-1} \left( \vec{z}\perp\vec{x} \land \bigwedge_{I\subseteq n}\{x_i \mid i\notin I\}\perp_{\{x_i\mid i\in I\}\vz}\{y_i \mid i\in I\} \right).
    \]
    Then the elimination rule of $\tilde\exists$ gives that $\tilde\exists z_0\exists z_1\dots\exists z_{l-1}\=(\vec z)$ and $\varphi$ entail $\psi$. As $\tilde\exists z_0\exists z_1\dots\exists z_{l-1}\=(\vec z)$ can be deduced with no assumptions, we obtain the desired deduction.

    Next we show that $\psi\entails\varphi$.
    Given $I\subseteq n$, consider the assumptions $\vec{z}\perp\vec{x}$ and $\{ x_i \mid i\notin I \}\perp_{\{x_i\mid i\in I\}\vec{z}}\{y_i \mid i\in I\}$. By Weak Union, Permutation and Symmetry, from $\vec z\perp \vec x$ we obtain $\{ x_i \mid i\notin I \}\perp_{\{ x_i \mid i\in I \}}\vec z$. From $\{ x_i \mid i\notin I \}\perp_{\{ x_i \mid i\in I \}}\vec z$ and $\{x_i \mid i\notin I\}\perp_{\{x_i\mid i\in I\}\vec{z}}\{y_i\mid i\in I\}$ Contraction gives $\{x_i \mid i\notin I\}\perp_{\{x_i\mid i\in I\}}\{y_i\mid i\in I\}\vz$, whence the Weakening Rule yields $\{x_i \mid i\notin I\}\perp_{\{x_i\mid i\in I\}}\{y_i\mid i\in I\}$. Introducing conjunctions, $\vec{z}\perp\vec{x}$ together with $\bigwedge_{I\subseteq n} \{x_i \mid i\notin I\}\perp_{\{x_i\mid i\in I\}\vec{z}}\{y_i\mid i\in I\}$ entails $\bigwedge_{I\subseteq n} \{x_i \mid i\notin I\}\perp_{\{x_i\mid i\in I\}}\{y_i\mid i\in I\}$, so by elimination of existential quantifiers in $\exists z_0\dots\exists z_{l-1} \left( \vec{z}\perp\vec{x} \land \bigwedge_{I\subseteq n} \{x_i \mid i\notin I\}\perp_{\{x_i\mid i\in I\}\vec{z}}\{y_i\mid i\in I\} \right)$, we obtain the deduction $\exists z_0\dots\exists z_{l-1} \left( \vec{z}\perp\vec{x} \land \bigwedge_{I\subseteq n} \{x_i \mid i\notin I\}\perp_{\{x_i\mid i\in I\}\vec{z}}\{y_i\mid i\in I\} \right)\entails\varphi$. By the elimination rule of $\tilde\exists$, $\psi\entails\varphi$ finally follows.
\end{proof}

\begin{proposition}
    \label{Proposition: Every empirical team is realized by a hidden-variable team supporting strong determinism}
    The formula $\tilde{\exists}z_0\exists z_1\dots\exists z_{l-1}\bigwedge_{i<n}\=(x_i\vec{z},y_i)$ is valid. More generally, $\varphi\models\tilde{\exists}z_0\exists z_1\dots\exists z_{l-1}(\bigwedge_{i<n}\=(x_i\vec{z},y_i)\wedge\varphi)$, when $\vec{z}$ does not occur free in $\varphi$.
\end{proposition}

In words, every empirical team is realized by a hidden-variable team supporting strong determinism.

\begin{proof}
    Essentially proved in~\cite{MR3038040}.
\end{proof}

\begin{remark}\label{Remark: With sort quantifier is valid but without is not}
    Note that while the formula $\tilde{\exists}z_0\exists z_1\dots\exists z_{l-1}\bigwedge_{i<n}\=(x_i\vec{z},y_i)$ is valid, the formula $\exists z_0\dots\exists z_{l-1}\bigwedge_{i<n}\=(x_i\vec{z},y_i)$ may not be, as demonstrated by the simple counter-example in the case where $n = 4$ and $l=1$: the domain of the structure is $\{0,1\}$, and the team is the full team $\{0,1\}^{\Vm\cup\Vo}$. It would seem that this problem could be overcome by increasing the length of the hidden-variable tuple: a sufficient condition for $\exists z_0\dots\exists z_{l-1}\bigwedge_{i<n}\=(x_i\vec{z},y_i)$ to be satisfied is that $\vec z$ can be assigned enough values to make each value of $x_i\vec z$ unique for all $i<n$, whence $\=(x_i\vec{z},y_i)$ is trivially satisfied.
\end{remark}

\begin{proposition}
    The formula $\tilde{\exists}z_0\exists z_1\dots\exists z_{l-1}(\=(\vec{x}\vec{z},\vec{y})\land \vec{z}\perp\vec{x})$ is valid. More generally, $\varphi\models\tilde{\exists}z_0\exists z_1\dots\exists z_{l-1}(\=(\vec{x}\vec{z},\vec{y})\land \vec{z}\perp\vec{x}\wedge\varphi)$, when $\vec{z}$ does not occur free in $\varphi$.
\end{proposition}

In words, every empirical team is realized by a hidden-variable team supporting weak determinism and $\vz$-independence.

\begin{proof}
    Essentially proved in~\cite{MR3038040}.
\end{proof}

\begin{proposition}
    \label{Proposition: hidden-variable team supporting z-independence and locality is equivalent to a hidden-variable team supporting z-independence and strong determinism}
    Let
    \begin{align*}
        \varphi &= \bigwedge_{I\subseteq n}\left( \{x_i \mid i\notin I\}\perp_{\{x_i\mid i\in I\}\vec{z}}\{y_i\mid i\in I\} \right), \\
        \psi &= \bigwedge_{i<n}\left( y_i\perp_{\vec{x}\vec{z}}\{y_j \mid j\neq i \} \right) \text{ and} \\
        \theta &= \bigwedge_{i<n} \=(x_i\vec{z},y_i).
    \end{align*}
    Then $\tilde{\exists}z_0\exists z_1\dots\exists z_{l-1}(\vz\perp\vx\land\varphi\land\psi) \models \tilde{\exists}z_0\exists z_1\dots\exists z_{l-1}(\vz\perp\vx\land\theta)$.
\end{proposition}

In words, any hidden-variable team supporting $\vz$-independence and locality is equivalent (in the sense of Definition~\ref{equivalent}) to a hidden-variable team supporting $\vz$-independence and strong determinism.

\begin{proof}
    Essentially proved in~\cite{MR3038040}.
\end{proof}

We do not know whether the logical consequences of Propositions~\ref{Proposition: Every empirical team is realized by a hidden-variable team supporting strong determinism}--\ref{Proposition: hidden-variable team supporting z-independence and locality is equivalent to a hidden-variable team supporting z-independence and strong determinism} are provable from our axioms. This is a subject of further study. Remark~\ref{Remark: With sort quantifier is valid but without is not} would suggest that stronger axioms for $\tilde\exists$ be required, as the semantic proofs of these propositions make use of the possibility of acquiring values for the hidden variables $\vz$ outside of the values that are possible for $\vx$ and $\vy$.

\subsection{Representation of No-Go Theorems in Team Semantics}

We now turn to the representation of no-go theorems in the foundations of quantum mechanics in terms of team semantics.
These results have fundamental significance, both foundationally, and also for their implications for quantum information and computation.
They rule out the possibility, even in principle, of accounting for quantum behaviour by means of local hidden-variable theories. This shows that the behaviour of quantum mechanics is essentially and unavoidably \emph{non-local}. This non-locality is, on one hand, highly challenging in terms of understanding what quantum mechanics is telling us about the nature of physical reality. On the other hand, this non-classicality opens up the possibility of performing information processing tasks using quantum resources which provably exceed what can be done classically.

How can these no-go theorems be represented in terms of team semantics?
All the results so far are examples of logical consequences and equivalences between team properties (or formulas). 
To prove the no-go theorems, we shall exhibit some \emph{counter-example teams} that demonstrate certain failures of logical consequence. 
These failures will imply the impossibility of describing these teams in terms of local hidden variables. 
As we shall see later, these teams do arise as behaviours of certain quantum mechanical systems.

The original result in this line is the celebrated Bell's Theorem \cite{MR3790629}.
However, that result in its original form is probabilistic in character, and hence not amenable to formalization in terms of team semantics.\footnote{It can be formalized in terms of probabilistic teams, which we will study in the next section.}
Two later constructions, due to Greenberger--Horne--Zeilinger (GHZ)~\cite{MR1081993,liu-et-al} and Hardy~\cite{PhysRevLett.71.1665}, strengthen Bell's result by constructions which work at the purely possibilistic level, and hence can be formalized directly in team semantics.

In Tarski semantics of first-order logic, some existential formulas, such as $\exists z (x=z \land \neg y=z)$, are not valid while others, such as $\exists z (x=z \lor x=y)$, are. To decide which are valid and which are not is particularly simple, especially in the empty vocabulary because first-order logic has in that case elimination of quantifiers. In team semantics where such quantifier elimination is not known to be possible, non-valid existential-conjunctive formulas can be quite complicated, as the examples below show. As we shall see, the no-go results of quantum mechanics give rise to very interesting teams.

We turn firstly to the GHZ construction.

\begin{definition}
    Assume that $n=3$.
    Let $X$ be an empirical team with $\rng(X) = \{0,1\}$. Denote
    \begin{align*}
        P &= \{(0,1,1), (1,0,1), (1,1,0)\}, \\
        Q &= \{(0,0,0), (0,1,1), (1,0,1), (1,1,0)\} \text{ and}\\
        R &= \{(0,0,1), (0,1,0), (1,0,0), (1,1,1)\}.
    \end{align*}
    We say that $X$ is a GHZ team if it satisfies the following conditions.
    \begin{enumerate}
        \item $Q = \{s(\vec{y}) \mid s\in X, s(\vec{x})\in P\}$ and $P \subseteq \{s(\vec{x}) \mid s\in X, s(\vec{y})\in Q\}$.
        \item $R = \{s(\vec{y}) \mid s\in X, s(\vec{x}) = (0,0,0)\}$.
    \end{enumerate}
\end{definition}

The following is a minimal example of a GHZ team:

\begin{center}
    \begin{tabular}{c|cccccc}
        & $x_0$ & $x_1$ & $x_2$ & $y_0$ & $y_1$ & $y_2$ \\
        \hline
        $s_0$ & 0     & 0     & 0     & 0     & 0     & 1     \\
        $s_1$ & 0     & 0     & 0     & 0     & 1     & 0     \\
        $s_2$ & 0     & 0     & 0     & 1     & 0     & 0     \\
        $s_3$ & 0     & 0     & 0     & 1     & 1     & 1     \\
        $s_4$ & 0     & 1     & 1     & 0     & 0     & 0     \\
        $s_5$ & 0     & 1     & 1     & 0     & 1     & 1     \\
        $s_6$ & 1     & 0     & 1     & 1     & 0     & 1     \\
        $s_7$ & 1     & 1     & 0     & 1     & 1     & 0
    \end{tabular}
\end{center}

The following is like Proposition 6.2 in \cite{MR3038040}:
\begin{proposition}
    \label{Proposition: GHZ} The formula $\tilde{\exists}z_0\exists z_1\dots\exists z_{l-1}(\vz\perp\vx\land\varphi\land\psi)$, where
    \begin{align*}
        \varphi &= \bigwedge_{I\subseteq n}\left( \{x_i \mid i\notin I\}\perp_{\{x_i\mid i\in I\}\vec{z}}\{y_i\mid i\in I\} \right) \text{ and} \\
        \psi &= \bigwedge_{i<n}\left( y_i\perp_{\vec{x}\vec{z}}\{y_j \mid j\neq i \} \right),
    \end{align*}
    is not valid, as demonstrated by any GHZ team.
\end{proposition}

In words, no GHZ team can be realized by a hidden-variable team supporting $\vz$-independence and locality.

\begin{proof}
    Proved in~\cite{MR3038040}.
\end{proof}

Next, we consider the Hardy construction. 

\begin{definition}
    Assume that $n=2$.
    Let $X$ be an empirical team with $\rng(X) = \{0,1\}$. Let $s_0,\dots,s_3$ be as in the following table.
    \begin{center}
        \begin{tabular}{c|cccc}
            & $x_0$ & $x_1$ & $y_0$ & $y_1$ \\
            \hline
            $s_0$ & $0$   & $0$   & $0$   & $0$   \\
            $s_1$ & $0$   & $1$   & $0$   & $0$   \\
            $s_2$ & $1$   & $0$   & $0$   & $0$   \\
            $s_3$ & $1$   & $1$   & $1$   & $1$
        \end{tabular}
    \end{center}
    We say that $X$ is a Hardy team if the following hold:
    \begin{enumerate}
        \item $s_0\in X$ but $s_1,s_2,s_3\notin X$, and \label{item: Hardy requirement 1}
        \item for every pair $\vec{a}\in\{0,1\}^2$ there is some $s\in X$ with $s(\vec{x})=\vec{a}$.\footnote{I.e. the team satisfies the so-called \emph{universality atom} $\mathop{\forall}(\vec{x})$.}
    \end{enumerate}
\end{definition}

A minimal example of a Hardy team would be the following:

\begin{center}
    \begin{tabular}{c|cccc}
        & $x_0$ & $x_1$ & $y_0$ & $y_1$ \\
        \hline
        $s_0$  & $0$   & $0$   & $0$   & $0$   \\
        $s_1'$ & $0$   & $1$   & $1$   & $1$   \\
        $s_2'$ & $1$   & $0$   & $1$   & $1$   \\
        $s_3'$ & $1$   & $1$   & $0$   & $0$
    \end{tabular}
\end{center}

The following is like Proposition 6.3 in \cite{MR3038040}:
\begin{proposition}
    \label{Proposition: Hardy}
    No Hardy team can be realized by a hidden-variable team supporting $\vz$-independence and locality.
\end{proposition}

Theorem~\ref{Proposition: Hardy} gives an alternative proof that the formula in Theorem~\ref{Proposition: GHZ} is not valid.

\begin{proof}[Proof of Proposition~\ref{Proposition: Hardy}]
    Proved in~\cite{MR3038040}.
\end{proof}

\paragraph{Discussion.}
We emphasize that the choice of specific counter-example teams for these results is significant, since to apply the results to quantum mechanics, we must show that these specific teams can be realized in quantum mechanics. We shall discuss quantum-realizability in Section~5.

Our reason for discussing both the GHZ and Hardy constructions is that they exhibit different ``strengths'' of non-locality. The GHZ construction exhibits a maximal form of non-locality; note that it requires at least a tripartite system (the construction can be generalized straightforwardly to $n$-partite systems for $n > 3$). The Hardy construction only requires a bipartite system, but exhibits a weaker form of non-locality. Both are stronger than the original probabilistic form of Bell's theorem in \cite{MR3790629}. For a detailed discussion of this hierarchy, see \cite{Abramsky_2011}.

The Kochen--Specker construction~\cite{MR0219280} gives an example of an empirical model which cannot be realized by any hidden-variable model supporting $\vz$-independence and parameter independence, providing a result even stronger than Propositions \ref{Proposition: GHZ} and~\ref{Proposition: Hardy}. However, the model in question does not quite fit our framework, which deals with so-called Bell-type scenarios. A sheaf-theoretic framework is given in~\cite{Abramsky_2011}, which subsumes both non-locality arguments for Bell scenarios, and contextuality proofs exemplified by  the Kochen--Specker construction.
This
could be translated into the language of team semantics using some version of the polyteam semantics of~\cite{MR3763187}.
A polyteam is essentially a set of teams. A simple example is the combination of a team describing lecture courses in an academic department, with variables for course name and course lecturer, and a different personnel team with variables for lectures and their office hours. Polyteams seem to be a suitable framework which allows a team semantics analysis of the sheaf-theoretic framework of \cite{Abramsky_2011}, and hence enables a treatment of the
Kochen-Specker Theorem and related results.
We shall leave the elaboration of this idea to future work.

\section{Independence Logic in Probabilistic and $K$-Team Semantics}\label{sec:probabilistic-teams}

We show that the probabilistic framework of Brandenburger and Yanofsky~\cite{MR2443068} can be translated to the language of probabilistic team semantics exactly the same way that the purely relational framework of~\cite{MR3038040} translates to ordinary team semantics. All the formal proofs of the relational setting turn out to be sound also in the probabilistic framework, as we are able to prove the validity of our axioms also in this setting, and even in $K$-team semantics for $K$ a positive, commutative and multiplicatively cancellative semiring. A semiring is multiplicatively cancellative if $ab = ac$ implies $b = c$ whenever $a \neq 0$. A result of Hannula~\cite{hannula2023conditional} shows that the so-called semigraphoid axioms are sound for such $K$, and Proposition~\ref{Proposition: Axioms of independence atom follow from semigraphoid axioms} shows that our axioms are all provable from them.

\subsection{Probabilistic Teams}

So far we have only been looking at possibilistic (i.e. two-valued relational) versions of the independence notions of quantum physics while these notions are usually taken to be probabilistic. To be able to discuss the probabilistic notions from the point of view of team semantics, we need a suitable framework. For this, we consider \emph{probabilistic team semantics}.

The study of a probabilistic variant of independence logic was first done in a multiteam setting in~\cite{MR3833654}. Prior to that, multiteams were studied in~\cite{MR3423958},~\cite{MR3682843} and~\cite{MR3587733}. \emph{Probabilistic teams} were then introduced by Durand et al. in~\cite{MR3802381} as a way to generalize multiteams, and further investigated in~\cite{MR3964038}. They can be thought of as a special case of \emph{measure teams}, another approach to probabilities in team semantics given in~\cite{MR3682843}.

It should be observed that we are not introducing probabilistic logic in the sense of formulas having probabilities. In our approach, only the teams are probabilistic and the logic is two-valued.

\begin{definition}
    Let $A$ be a finite set and $V$ a finite set of variables. A \emph{probabilistic team}, with \emph{variable domain} $V$ and \emph{value domain} $A$, is a probability distribution $\X\colon A^V \to [0,1]$.

    Let $\A$ be a (possibly many-sorted) finite structure, and let $X$ be the full team of $\A$ with domain $V$. Then a probabilistic team of $\A$ with variable domain $V$ is any distribution $\X\colon X\to[0,1]$.
\end{definition}

Ordinary teams of size $k$ can be seen as probabilistic teams by giving each assignment in the team probability $1/k$ and assignments not in the team probability zero. This idea of treating ordinary teams as uniformly distributed probabilistic teams generalizes to multiteams i.e. teams in which assignments can have several occurrences. Then an assignment which occurs $m$ times is given the probability $m/k$. In fact, it is not difficult to see that any probabilistic team with rational probabilities corresponds to a multiteam.

We will call a probabilistic team with variable domain $\Vm\cup\Vo$ a \emph{probabilistic empirical team} and a probabilistic team with variable domain $\Vm\cup\Vo\cup\Vh$ a \emph{probabilistic hidden-variable team}.

\begin{definition}
    We say that a team $X$ is the \emph{possibilistic collapse} of a probabilistic team $\X$ if for any assignment $s$, $s\in X$ if and only if $\X(s)>0$.
\end{definition}

Note that if $\X$ is a probabilistic team of $\A$, then the possibilistic collapse $X$ is a team of $\A$.

We may consider a probabilistic team a ``probabilistic realization'' of its collapse. Of course, an ordinary team has a multitude of such probabilistic realizations.

The possibilistic collapse of a probabilistic team $\X$ is also called the \emph{support} of $\X$ and denoted by $\support\X$.

We denote by $\abs{\X_{\vec{u}=\vec{a}}}$ the number
\[
    \sum_{\substack{s(\vec{u})=\vec{a} \\ s\in\support\X}}\X(s),
\]
i.e. the marginal probability of the variable tuple $\vec{u}$ having the value $\vec{a}$ in $\X$.

Next we define the probabilistic analogue for an empirical team being realized by a hidden-variable team.

\begin{definition}
    A probabilistic hidden-variable team $\Y$ \emph{realizes} a probabilistic empirical team $\X$ if for all $\vec{a}$ and $\vec{b}$ we have
    \begin{enumerate}
        \item $\abs{\X_{\vec{x}=\vec{a}}} = 0$ if and only if $\abs{\Y_{\vec{x}=\vec{a}}} = 0$, and
        \item $\abs{\X_{\vec{x}\vec{y}=\vec{a}\vec{b}}}\cdot\abs{\Y_{\vec{x}=\vec{a}}} = \abs{\Y_{\vec{x}\vec{y}=\vec{a}\vec{b}}}\cdot\abs{\X_{\vec{x}=\vec{a}}}$.
    \end{enumerate}
    $\Y$ \emph{uniformly realizes} $\X$ if in addition $\abs{\X_{\vec{x}=\vec{a}}}=\abs{\Y_{\vec{x}=\vec{a}}}$ for all $\vec{a}$.
\end{definition}

The intuition behind the definition is the following: $\Y$ realizes $\X$ if the probability of the event ``$s(\vec{x})=\vec{a}\,$'' is non-zero in both teams exactly the same time, and in case the probability indeed is non-zero, the probability of the event ``$s(\vec{y})=\vec{b}\,$'', conditional to ``$s(\vec{x})=\vec{a}\,$'', is the same in both teams. Uniform realizability appears to be a stronger concept than realizability, but we do not have an example for that yet.

\begin{proposition}
    \label{Proposition: Probabilistic realization implies possibilistic realization}
    If $\Y$ realizes $\X$, then the possibilistic collapse of $\Y$ realizes the possibilistic collapse of $\X$.
\end{proposition}

Proposition~\ref{Proposition: Probabilistic realization implies possibilistic realization} says that one obtains the same team by first projecting away hidden variables and then taking the possibilistic collapse as one gets by first taking the possibilistic collapse and then projecting away the hidden variables, i.e. the diagram in Figure~\ref{Figure: Probabilistic realization implies possibilistic realization} commutes.

\begin{figure}[htb]
    \begin{center}
        \begin{tikzpicture}[scale=1]
            \draw[black!33] (-3.5,0) -- (-6,4);
            \draw[black!33] (-3.5,0) -- (-1,4);
            \draw[black!33] (-6,4) -- (-1,4);
            \draw (-3.5,4.5) node[above] {Prob. teams};
            \draw[black!33] (3.5,0) -- (6,4);
            \draw[black!33] (3.5,0) -- (1,4);
            \draw[black!33] (6,4) -- (1,4);
            \draw (3.5,4.5) node[above] {Teams};
            \fill[black] (-4,3.5) circle (0.075) edge[->, bend left = 10, shorten >= 4pt] (4,3.5) edge[->, shorten >= 4pt] (-3.5,1.5) node[left] {$\Y$};
            \fill[black] (-3.5,1.5) circle (0.075) edge[->, bend left = 10, shorten >= 4pt] (3.5,1.5) node[below] {$\X$};
            \draw (0, 4) node[above] {\footnotesize collapse};
            \draw (0, 2) node[above] {\footnotesize collapse};
            \draw (-1.8,2.75) node[left] {\footnotesize projection};
            \draw (1.85,2.75) node[right] {\footnotesize projection};
            \fill[black] (4,3.5) circle (0.075) edge[->, shorten >= 4pt] (3.5,1.5) node[right] {$Y$};
            \fill[black] (3.5,1.5) circle (0.075) node[below] {$X$};
        \end{tikzpicture}
    \end{center}
    \caption{Probabilistic realization implies possibilistic realization.}
    \label{Figure: Probabilistic realization implies possibilistic realization}
\end{figure}
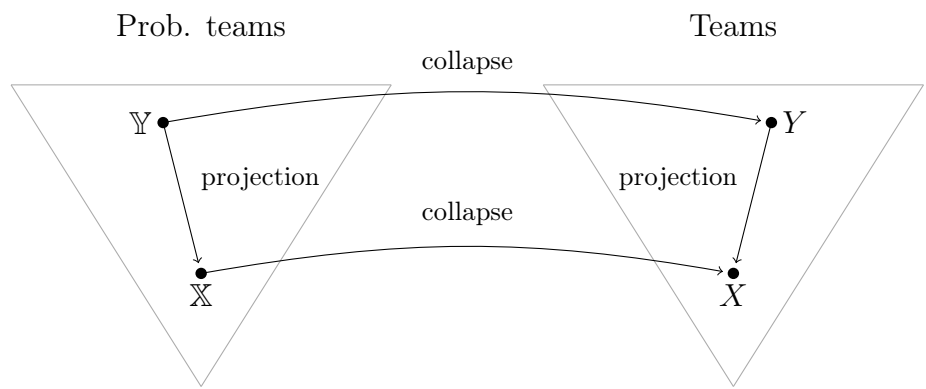

\begin{proof}[Proof of Proposition~\ref{Proposition: Probabilistic realization implies possibilistic realization}]
    Suppose that $\Y$ realizes $\X$, and denote by $Y$ and $X$ the respective possibilistic collapse. In order to prove that $Y$ realizes $X$, we need to show that for all assignments $s$,
    \[
        s\in X \iff \exists s'\in Y \bigwedge_{i<n}(s'(x_i)=s(x_i) \land s'(y_i)=s(y_i)).
    \]
    We show only one direction, the other one is similar.

    Suppose that $s\in X$. Denote $\vec{a}=s(\vec{x})$ and $\vec{b}=s(\vec{y})$. The aim is to show that there is some $s'\in Y$ with $s'(\vec{x})=\vec{a}$ and $s'(\vec{y})=\vec{b}$. Since $X$ is the possibilistic collapse of $\X$ and $s\in X$, we have $\X(s)>0$, and as in addition $s(\vec{x})=\vec{a}$, we obtain $\abs{\X_{\vec{x}=\vec{a}}} > 0$. Thus, as $\Y$ realizes $\X$,  $\abs{\Y_{\vec{x}=\vec{a}}} > 0$. Similarly, $\abs{\X_{\vec{x}\vec{y}=\vec{a}\vec{b}}} > 0$, and as
    \[
        \abs{\Y_{\vec{x}\vec{y}=\vec{a}\vec{b}}} = \frac{\abs{\X_{\vec{x}\vec{y}=\vec{a}\vec{b}}}\cdot\abs{\Y_{\vec{x}=\vec{a}}}}{\abs{\X_{\vec{x}=\vec{a}}}},
    \]
    also $\abs{\Y_{\vec{x}\vec{y}=\vec{a}\vec{b}}} > 0$. This means that there is some $s'$ with $\Y(s') > 0$ and $s'(\vec{x}\vec{y})=\vec{a}\vec{b}$. Since $Y$ is the possibilistic collapse of $\Y$, this means that $s'\in Y$, as desired.
\end{proof}

\subsection{Probabilistic independence logic}\label{subsec:probabilistic-independence-logic}

We now present the semantics of the \emph{probabilistic (conditional) independence atom} $\vec{u}\probind_{\vec{v}}\vec{w}$, as defined in~\cite{MR3802381}.

\begin{definition}
    Let $\A$ be a structure and $\X$ a probabilistic team of $\A$, and let $\vec{u}$, $\vec{v}$ and $\vec{w}$ be tuples of variables. Then $\X$ satisfies the formula $\vec{u}\probind_{\vec{v}}\vec{w}$ in $\A$, in symbols $\A\models_{\X}\vec{u}\probind_{\vec{v}}\vec{w}$, if for all $\vec{a}$, $\vec{b}$ and $\vec{c}$,
    \[
        \abs{\X_{\vec{u}\vec{v}=\vec{a}\vec{b}}} \cdot \abs{\X_{\vec{v}\vec{w}=\vec{b}\vec{c}}} = \abs{\X_{\vec{u}\vec{v}\vec{w}=\vec{a}\vec{b}\vec{c}}} \cdot \abs{\X_{\vec{v}=\vec{b}}}.
    \]
\end{definition}

The intention behind the atom is to capture the notion of conditional independence in probability theory: denoting a probability measure by $p$, two events $A$ and $B$ are conditionally independent over an event $C$ (assuming $p(C)>0$) if
\[
    p(A\mid C)p(B\mid C) = p(A\cap B\mid C).
\]
Recalling that $p(D\mid C) = p(D\cap C)/p(C)$, we can multiply both sides of the equation by $p(C)^2$ and obtain
\[
    p(A\cap C)p(B\cap C) = p(A\cap B\cap C)p(C),
\]
which is exactly what the probabilistic dependence atom expresses. In the case when $p(C)=0$, both sides of the new equation are $0$, so in that case the independence atom is vacuously true.

Exactly the same way as in ordinary independence logic, we can define the \emph{probabilistic dependence atom} $\=(\vec{v},\vec{w})$ via the independence atom:

\begin{definition}
    Let $\vec{v}$ and $\vec{w}$ tuples of variables. Then by $\=(\vec{v},\vec{w})$ we mean the formula $\vec{w}\probind_{\vec{v}}\vec{w}$.
    By $\=(\vec{w})$ we mean the formula $\vec{w}\probind\vec{w}$ and call it the \emph{probabilistic constancy atom}.
\end{definition}

The syntax of probabilistic independence logic is the same as the syntax of ordinary independence logic, except that we use the symbol $\probind$ instead of $\perp$. Next we present the semantics of more complex formulas of probabilistic independence logic, as defined in~\cite{MR3802381}. First we define the \emph{$r$-scaled union} $\X\sqcup_r\Y$ of two probabilistic teams $\X$ and $\Y$ with the same variable and value domain, for $r\in[0,1]$, by setting
\[
    (\X\sqcup_r\Y)(s) \coloneqq r\X(s) + (1-r)\Y(s).
\]
We define the (``duplicated'') team $\X[A_{\sort(v)}/v]$ by setting
\[
    \X[A_{\sort(v)}/v](s(a/v)) \coloneqq \sum_{\substack{t\in\support\X \\ t(a/v)=s(a/v)}}\frac{\X(t)}{\abs{A_{\sort(v)}}}
\]
for all $a\in A_{\sort(v)}$. If $v$ is a fresh variable, i.e. not in the variable domain of $\X$, then $\X[A/v](s(a/v))=\X(s)/\abs{A_{\sort(v)}}$ for all $s\in\support\X$. Finally, given a function $F$ from the set $\support\X$ to the set of all probability distributions on $A_{\sort(v)}$, we define the (``supplemented'') team $\X[F/v]$ by setting
\[
    \X[F/v](s(a/v)) \coloneqq \sum_{\substack{t\in\support\X \\ t(a/v)=s(a/v)}} \X(t)F(t)(a)
\]
for all $a\in A_{\sort(v)}$. Again, if $v$ is fresh, then $\X[F/v](s(a/v)) = \X(s)F(s)(a)$. It is easy to see that both duplication and supplementation give rise to well-defined probabilistic teams.

\begin{definition}
    Let $\A$ be a structure and $\X$ a probabilistic team of $\A$. Then
    \begin{enumerate}
        \item $\A\models_{\X} \alpha$ for a first-order atomic or negated atomic formula $\alpha$ if $\A\models_X\alpha$, where $X$ is the possibilistic collapse of $\X$.
        \item $\A\models_{\X}\varphi\land\psi$ if $\A\models_{\X}\varphi$ and $\A\models_{\X}\psi$.
        \item $\A\models_{\X}\varphi\lor\psi$ if $\A\models_{\Y}\varphi$ and $\A\models_{\Z}\psi$ for some probabilistic teams $\Y$ and $\Z$, and $r\in[0,1]$ such that $\X=\Y\sqcup_r\Z$.
        \item $\A\models_{\X}\forall v\varphi$ if $\A\models_{\X[A_{\sort(v)}]}\models\varphi$.
        \item $\A\models_{\X}\exists v\varphi$ if $\A\models_{\X[F/v]}\models \varphi$ for some function $F\colon\support\X\to\{p\in[0,1]^{A_{\sort(v)}} \mid \text{$p$ is a probability distribution}\}$.
        \item $\A\models_{\X}\tilde{\forall}v\varphi$ if $\B\models_{\X}\forall v\varphi$ for all expansions $\B$ of $\A$ by the sort $\sort(v)$.
        \item $\A\models_{\X}\tilde{\exists}v\varphi$ if $\B\models_{\X}\exists v\varphi$ for some expansion $\B$ of $\A$ by the sort $\sort(v)$.
    \end{enumerate}
\end{definition}

Again, when it is clear what is meant, we write $\X\models\varphi$ instead of $\A\models_{\X}\varphi$.

By definition, first-order atomic formulas are satisfied by a probabilistic team if and only if the underlying possibilistic collapse satisfies them. This property is in the multiteam setting of~\cite{MR3833654} called weak flatness:

\begin{definition}
    We say that a formula $\varphi$ of probabilistic independence logic is \emph{weakly flat} if for all probabilistic teams $\X$, we have
    \[
        \X\models\varphi \iff \support\X\models\varphi^*,
    \]
    where $\varphi^*$ is the formula of independence logic obtained from $\varphi$ by replacing each occurrence of the symbol $\probind$ by the symbol $\perp$.
    A sublogic of probabilistic independence logic is weakly flat if every formula of the logic is.
\end{definition}

Later on, we simply write $\varphi$ instead of $\varphi^*$ whenever it is obvious what is meant.

\begin{lemma}
    The probabilistic dependence atom is weakly flat.
\end{lemma}
\begin{proof}
    The proof given in the multiteam setting in~\cite{MR3833654} works also in the probabilistic team setting.
\end{proof}

It turns out that one direction of the equivalence in the definition of weak flatness always holds:

\begin{proposition}
    \label{Proposition: Truth in independence logic is preserved by possibilistic collapse}
    Let $\X$ be a probabilistic team and $\varphi$ a formula of independence logic. Denote by $X$ the possibilistic collapse of $\X$. Then
    \[
        \X\models\varphi \implies X\models\varphi.
    \]
\end{proposition}
We omit the proof, as it is available in~\cite{2021arXiv210210931A}.

\begin{lemma}\label{Lemma: Logical operations preserve weak flatness}
    Logical operations of Definition~\ref{Definition: Syntax of independence logic} preserve weak flatness.
\end{lemma}
\begin{proof}
    Suppose that $\varphi$ and $\psi$ are weakly flat. We then show that the formulas $\varphi\land\psi$, $\varphi\lor\psi$, $\exists v\varphi$, $\forall v\varphi$, $\tilde{\exists}v\varphi$ and $\tilde{\forall}v\varphi$ are weakly flat. Let $\A$ be a structure and $\X$ a probabilistic team of $\A$, and let $X$ be the possibilistic collapse of $\X$. Note that to show that a formula $\theta$ is weakly flat, we only need to show that
    \[
        \A\models_X\theta \implies \A\models_{\X}\theta,
    \]
    as the other direction follows from Proposition~\ref{Proposition: Truth in independence logic is preserved by possibilistic collapse}.

    \begin{enumerate}
        \item The case for conjunction is trivial.

        \item Suppose that $X\models\varphi\lor\psi$. Then there are $X_0$ and $X_1$ such that $X_0\models\varphi$ and $X_1\models\psi$ and $X=X_0\cup X_1$. Let $\X_i$ be a probabilistic team with collapse $X_i$ such that
        \[
            \X_i(s) =
            \begin{cases}
                \X(s)/(2p_i+q) & \text{if $s\in X_0\cap X_1$,} \\
                \X(s)/(p_i+q/2) & \text{otherwise,}
            \end{cases}
        \]
        where $p_i=\sum_{s\in X_i\setminus X_{1-i}}\X(s)$ and $q = \sum_{s\in X_0\cap X_1}\X(s)$.
        As $\varphi$ and $\psi$ are weakly flat, $\X_0\models\varphi$ and $\X_1\models\psi$. Now, if $s\in X_0\setminus X_1$, then
        \begin{align*}
            \X(s) &= \frac{(p_0+q/2)\X(s)}{p_0+q/2} = (p_0+q/2)\X_0(s) \\
            &= (\X_0\sqcup_{p_0 + q/2}\X_1)(s),
        \end{align*}
        and if $s\in X_1\setminus X_0$, then
        \begin{align*}
            \X(s) &= \frac{(p_1+q/2)\X(s)}{p_1+q/2} = (p_1+q/2)\X_1(s) = (1 - (p_0 + q/2))\X_1(s) \\
            &= (\X_0\sqcup_{p_0 + q/2}\X_1)(s),
        \end{align*}
        and if $s\in X_0\cap X_1$, then
        \begin{align*}
            \X(s) &= \frac{\X(s)}{2} + \frac{\X(s)}{2} = \frac{(p_0 + q/2)\X(s)}{2(p_0 + q/2)} + \frac{(p_1 + q/2)\X(s)}{2(p_1 + q/2)} \\
            &= \frac{(p_0 + q/2)\X(s)}{2p_0 + q} + \frac{(p_1 + q/2)\X(s)}{2p_1 + q} \\
            &= (p_0 + q/2)\X_0(s) + (p_1 + q/2)\X_1(s) \\
            &= (p_0 + q/2)\X_0(s) + (1 - (p_0 + q/2))\X_1(s) \\
            &= (\X_0\sqcup_{p_0 + q/2}\X_1)(s).
        \end{align*}
        Hence $\X=\X_0\sqcup_{p_0 + q/2}\X_1$ and thus $\X\models\varphi\lor\psi$.
    
        \item Suppose that $X\models\exists v\varphi$. Then there is a function $F\colon X\to A_{\sort(v)}$ such that $X[F/v]\models\varphi$. Define a function $G\colon X\to\{p\in[0,1]^{A_{\sort(v)}} \mid \text{$p$ is a distribution}\}$ by setting
        \[
            G(s)(a) =
            \begin{cases}
                1/\abs{F(s)} & \text{if $a\in F(s)$,} \\
                0 & \text{otherwise.}
            \end{cases}
        \]
        Then $X[F/v]$ is the possibilistic collapse of $\X[G/v]$, as
        \begin{align*}
            \X[G/v](s(a/v)) > 0 &\iff \sum_{\substack{t\in X \\ t(a/v)=s(a/v)}}\X(t)G(t)(a) > 0 \\
            &\iff \exists t\in X\ (G(t)(a) > 0\ \text{and}\ t(a/v)=s(a/v)) \\
            &\iff \exists t\in X\ (a\in F(t)\ \text{and}\ t(a/v)=s(a/v)) \\
            &\iff s(a/v)\in X[F/v].
        \end{align*}
        Then as $\varphi$ is weakly flat, $\X[G/v]\models\varphi$, so $\X\models\varphi$.
    
        \item The universal quantifier case is similar.
    
        \item Suppose that $\A\models_{X}\tilde{\exists}v\varphi$. Then there is an expansion $\B$ of $\A$ of the new sort $\sort(v)$ such that $\B\models_{X}\exists v\varphi$. We already showed that $\exists v\varphi$ is weakly flat, so thus $\B\models_{\X}\exists v\varphi$. Thus $\A\models_{\X}\tilde{\exists}v\varphi$.
    
        \item The universal sort quantifier case is similar.
    \end{enumerate}
\end{proof}

In ordinary team semantics, the dependence atom is \emph{downwards closed}, meaning that if $X\models\=(\vec{v},\vec{w})$, then for any $Y\subseteq X$ also $Y\models\=(\vec{v},\vec{w})$. We define an analogous concept of downwards closedness and show that dependence logic is downwards closed also in probabilistic team semantics.

\begin{definition}
    We say that a probabilistic team $\Y$ is a \emph{weak subteam} of a probabilistic team $\X$ if they have the same variable and value domain and, denoting by $Y$ and $X$ the respective possibilistic collapses, $Y\subseteq X$. We say that $\Y$ is a \emph{subteam} of $\X$ if it is a weak subteam of $\X$ and there is $r\in(0,1]$ such that $\X(s) = r\Y(s)$ for all $s\in Y$.
\end{definition}

The concept of a weak subteam is the weakest notion of subteam that one would think of. Still, due to weak flatness, it seems to be enough.

\begin{definition}
    We say that a formula $\varphi$ of a probabilistic independence logic is \emph{downwards closed} if for all probabilistic teams $\X$ that satisfy $\varphi$, every subteam of $\X$ also satisfies $\varphi$. We say that a formula $\varphi$ is \emph{strongly downwards closed} if for all probabilistic teams $\X$ that satisfy $\varphi$, every weak subteam of $\X$ also satisfies $\varphi$. A sublogic of probabilistic independence logic is (strongly) downwards closed if every formula of the logic is.
\end{definition}

\begin{lemma}
    Every weakly flat formula that is downwards closed in ordinary team semantics is strongly downwards closed in probabilistic team semantics.
\end{lemma}
\begin{proof}
    Let $\varphi$ be a weakly flat formula that is downwards closed in ordinary team semantics. Let $\X$ a probabilistic team, $\Y$ a weak subteam of $\X$ and $X$ and $Y$ the respective possibilistic collapses. Suppose that $\X\models\varphi$. By weak flatness, $X\models\varphi$. By downwards closedness of $\varphi$ in ordinary team semantics, $Y\models\varphi$. Then by weak flatness again, $\Y\models\varphi$.
\end{proof}

Notice that all atomic formulas of probabilistic dependence logic are weakly flat, as proved in~\cite{MR3833654}. Hence we get the corollary:

\begin{corollary}
    \label{Corollary: Probabilistic dependence logic is weakly flat and strongly downwards closed}
    Probabilistic dependence logic is weakly flat and thus also strongly downwards closed.
\end{corollary}

Next we show that logical operations preserve downwards closedness. To make it easier, we first show that one may change the name of a bound variable without affecting the truth of the formula.

\begin{lemma}[Locality,~\cite{MR3802381}]
    \label{Lemma: Locality of probabilistic independence logic}
    Let $\varphi$ be a formula of probabilistic independence logic, with its free variables among $v_0,\dots,v_{m-1}$. Then for all probabilistic teams $\X$ whose variable domain $D$ includes the variables $v_i$ and any set $V$ such that $\{v_0,\dots,v_{m-1}\}\subseteq V\subseteq D$, we have
    \[
        \X\models\varphi \iff \X\restriction V\models\varphi,
    \]
    where $\X\restriction V$ is the probabilistic team $\Y$ with variable domain $V$ defined by $\Y(s)=\sum_{t\restriction V = s}\X(t)$.
\end{lemma}

\begin{lemma}
    \label{Lemma: The name of a bound variable can be changed}
    Let $\varphi$ be a formula of probabilistic independence logic, and let $Q\in\{\forall,\exists\}$. Then the formulas $Qv\varphi$ and $Qw\varphi(w/v)$ are equivalent, where $w$ is a variable that does not occur in $\varphi$ and $\varphi(w/v)$ denotes the formula one obtains by replacing every free occurence of variable $v$ in $\varphi$ by variabe $w$.
\end{lemma}
\begin{proof}
    The statement of the lemma easily follows from the following claim:
    \begin{quote}
        Let $\X$ and $\Y$ be probabilistic teams with variable domain $D_\X = \{v_0,\dots,v_{n-1}\}$ and $D_\Y = \{w_0,\dots,w_{n-1}\}$, respectively, such that
        \begin{enumerate}
            \item $\support\Y = \{s^* \mid s\in\support\X\}$, where $s^*$ is the assignment with domain $D_\Y$ such that $s^*(w_i) = s(v_i)$ for all $i<n$, and
            \item $\X(s) = \Y(s^*)$ for all $s\in\support\X$.
        \end{enumerate}
        Then for any $\varphi$ with free variables in $D_\X$,
        \[
            \X\models\varphi \iff \Y\models\varphi(w_0/v_0,\dots,w_{n-1}/v_{n-1}).
        \]
    \end{quote}
    The claim can be proved with a straightforward induction on $\varphi$.
\end{proof}

\begin{proposition}
    If all atomic formulas of a sublogic of probabilistic independence logic are (strongly) downwards closed, then the whole sublogic is.
\end{proposition}
\begin{proof}
    Suppose that all atomic formulas are downwards closed. We show by induction that every formula is.
    \begin{enumerate}
        \item The case of conjunction follows immediately from the induction hypothesis.

        \item Suppose that $\Y$ is a subteam of $\X$ and $\X\models\varphi\lor\psi$. Let $p\in(0,1]$ be such that $p\Y(s)=\X(s)$ for $s\in\support\Y$. Note that then $p=\sum_{s\in\support\Y}\X(s)$. Now there are $\X_0$ and $\X_1$ and $q\in[0,1]$ such that $\X_0\models\varphi$ and $\X_1\models\psi$ and $\X=\X_0\sqcup_q\X_1$. Then let $\Y_i$ be such that
        \[
            \support\Y_0 = \support\X_0\cap\support\Y \quad\text{and}\quad \Y_0(s) = \X_0(s)/p_0
        \]
        for $s\in\support\Y_i$, where $p_0 = \sum_{s\in\support\Y_0}\X_0(s)$, and
        \[
            \support\Y_1 = (\support\X_1\setminus\support\X_0)\cap\support\Y \quad\text{and}\quad \Y_1(s) = \X_1(s)/p_1
        \]
        for $s\in\support\Y_i$, where $p_1 = \sum_{s\in\support\Y_1}\X_1(s)$. Now
        \begin{enumerate}
            \item $\Y_i$ is well defined distribution for $i<2$, as
            \begin{align*}
                \sum_{s\in\support\Y_i}\Y_i(s) &= \sum_{s\in\support\Y_i}\X_i(s)/p_i = \frac{\sum_{s\in\support\Y_i}\X_i(s)}{\sum_{s\in\support\Y_i}\X_i(s)} = 1.
            \end{align*}
    
            \item $\Y_i$ is a subteam of $\X_i$ for $i<2$, as by definition $\support\Y_i\subseteq\support\X_i$ and $\X_i(s)=p_i\Y_i(s)$ for $s\in\support\Y_i$, where $p_i\in(0,1]$.
    
            \item $\Y = \Y_0\sqcup_{r}\Y_1$, where
            \[
                r = \frac{qp_0}{qp_0 + (1-q)p_1},
            \]
            as
            can be verified by a straightforward calculation.
        \end{enumerate}
        Then by the induction hypothesis, $\Y_0\models\varphi$ and $\Y_1\models\psi$, so $\Y\models\varphi\lor\psi$.

        \item Suppose that $\Y$ is a subteam of $\X$ and $\X\models\exists v\varphi$. Let $w$ be a fresh variable outside of the variable domain of $\X$. By Lemma~\ref{Lemma: The name of a bound variable can be changed}, $\X\models\exists w\varphi(w/v)$. Let $p\in(0,1]$ be such that $p\Y(s)=\X(s)$ for $s\in\support\Y$. Now $\X[F/w]\models\varphi(w/v)$ for some $F$. Let $G = F\restriction\support\Y$. Then $\Y[G/w]$ is a subteam of $\X[F/w]$, as
        \begin{align*}
            \X[F/w](s(a/w)) &= \X(s)F(s)(a) \\
            &= p\Y(s)G(s)(a) \\
            &= p\Y[G/w](s(a/w))
        \end{align*}
        for all $s\in\support\Y$. But then by the induction hypothesis, $\Y[F/w]\models\varphi(w/v)$, so $\Y\models\exists w\varphi(w/v)$. Then by Lemma~\ref{Lemma: The name of a bound variable can be changed}, $\Y\models\varphi$.
    
        \item The other quantifier cases are similar.
    \end{enumerate}
    Then suppose that all atomic formulas are strongly downwards closed.
    \begin{enumerate}
        \item The case for conjunction is again trivial.

        \item Suppose that $\Y$ is a weak subteam of $\X$ and $\X\models\varphi\lor\psi$. Then there is $r\in[0,1]$ and probabilistic teams $\X_0$ and $\X_1$ such that $\X = \X_0\sqcup_r\X_1$, $\X_0\models\varphi$ and $\X_1\models\psi$. Using an argument similar to the one presented in the proof of Lemma~\ref{Lemma: Logical operations preserve weak flatness}, we can define $\Y_0$, $\Y_1$ and $r'$ such that $\Y = \Y_0\sqcup_{r'}\Y_1$, $\support\Y_0 = \support\X_0\cap\support\Y$ and $\support\Y_1 = \support\X_1\cap\support\Y$. Now, as $\support\Y_i\subseteq\support\X_i$, $\Y_i$ is a weak subteam of $\X_i$ for $i=0,1$. Hence,  by the induction hypothesis, we have $\Y_0\models\varphi$ and $\Y_1\models\psi$. Hence $\Y\models\varphi\lor\psi$.

        \item Suppose that $\Y$ is a weak subteam of $\X$ and $\X\models\exists v\varphi$. Without loss of generality, $v$ is not in the variable domain of $\X$. Now $\X[F/v]\models\varphi$ for some $F$. Let $G = F\restriction\support\Y$. Then $\Y[G/v]$ is a weak subteam of $\X[F/v]$: for all assignments $s$ whose domain is the variable domain of $\X$ and elements $a$ from the value domain of $\X$,
        \begin{align*}
            s(a/v)\in\support\Y[G/v] &\implies \text{$s\in\support\Y$ and $G(s)(a)>0$} \\
            &\implies \text{$s\in\support\X$ and $F(s)(a)>0$} \\
            &\implies s(a/v)\in\support\X[F/v].
        \end{align*}
        By the induction hypothesis, $\Y[G/v]\models\varphi$, whence $\Y\models\exists v\varphi$.

        \item The other quantifier cases are similar.
    \end{enumerate}
\end{proof}

One might wonder whether ordinary independence logic is a result of ``collapsing'' probabilistic independence logic in the sense that a team $X$ satisfies a formula $\varphi$ if and only if it is the collapse of some probabilistic team $\X$ such that $\X$ satisfies the probabilistic version of $\varphi$. It turns out, in Proposition~\ref{Proposition: Truth in independence logic is preserved by possibilistic collapse}, that, indeed, given a probabilistic team that satisfies a formula, also the collapse will satisfy the (possibilistic version of the) formula. But given an ordinary team that satisfies a formula, there may not be \emph{any} probabilistic realization of that team that would satisfy the (probabilistic version of the) formula. We will see in Proposition~\ref{Proposition: Possibilistic no-signalling does not imply probabilistic no-signalling} that such a formula and a team can be quite simple.

We add a new operation $\PR$ to ordinary independence logic, defined by
\begin{center}
    $X\models\PR\varphi$ if there is a probabilistic team $\X$ such that $\X\models\varphi$ and $X = \support\X$.
\end{center}
One can then ask whether this operation is downwards closed, closed under unions, $\Sigma_1^1$-definable etc.

By weak flatness, for any formula $\varphi$ of \emph{dependence logic}, we have $\varphi\equiv\PR\varphi$.

A similar discussion can be applied to logical consequence. Logical consequence in ordinary team semantics is different from logical consequence in probabilistic team semantics: as was shown by Studený~\cite{Studeny90conditionalindependence} and also pointed out in the team semantics context by Albert \& Grädel~\cite{2021arXiv210210931A}, the following is an example of a rule that is sound in team semantics but not in probabilistic team semantics:
\[
    \{\vx\perp_{\vz}\vy,\vx\perp_{\vu}\vy, \vz\perp_{\vx\vy}\vu\} \quad\text{entails}\quad \vx\perp_{\vz\vu}\vy,
\]
while the following is an example of a rule that is sound in probabilistic team semantics but not in ordinary team semantics:
\[
    \{\vx\probind_{\vy\vz}\vy, \vz\probind_{\vx}\vu, \vz\probind_{\vy}\vu, \vx\probind\vy\} \quad\text{entails}\quad \vz\probind\vu.
\]
Certainly this means that such rules cannot be derived from the axioms of Definition~\ref{axioms} (or even the semigraphoid axioms), since the axioms are satisfied by both ordinary and probabilistic independence logic, as will be shown in Proposition~\ref{Proposition: Probabilistic soundness theorem}.

It was recently proved that the implication problem for probabilistic independence atoms is undecidable~\cite{MR4544911}.

Earlier we noted that in ordinary team semantics, being a realization of a hidden-variable team is expressible by means of  existential quantifiers. We show that this is also the case in probabilistic team semantics.

\begin{lemma}\label{419}
    Let $\A$ be a structure and $\B$ an expansion of $\A$ by the hidden-variable sort. Let $\X$ a probabilistic empirical team of $\A$ and $\Y$ a probabilistic hidden-variable team of $\B$. Then $\X$ is uniformly realized by $\Y$ if and only if
    \[
        \Y=\X[F_0/z_0]\dots[F_0/z_{l-1}]
    \]
    for some functions $F_i$.
\end{lemma}
\begin{proof}
    For simplicity we assume that $l=1$ and $\vec{z}=z$.

    Suppose that $\Y$ uniformly realizes $\X$. Then for all $\vec{a}$ and $\vec{b}$,
    \begin{enumerate}
        \item $\abs{\X_{\vec{x}=\vec{a}}}=\abs{\Y_{\vec{x}=\vec{a}}}$, and
        \item $\abs{\X_{\vec{x}\vec{y}=\vec{a}\vec{b}}}\abs{\Y_{\vec{x}=\vec{a}}} = \abs{\Y_{\vec{x}\vec{y}=\vec{a}\vec{b}}}\abs{\X_{\vec{x}=\vec{a}}}$.
    \end{enumerate}
    Define a function $F$ by setting
    \[
        F(s)(\gamma) = \frac{\Y(s(\gamma/z))}{\abs{\Y_{\vec{x}\vec{y}=s(\vec{x}\vec{y})}}}
    \]
    for all $s\in\support\X$ and $\gamma$. Now $F(s)$ is a distribution, as
    \[
        \sum_\gamma F(s)(\gamma) = \sum_\gamma \frac{\Y(s(\gamma/z))}{\abs{\Y_{\vec{x}\vec{y}=s(\vec{x}\vec{y})}}} = \frac{\abs{\Y_{\vec{x}\vec{y}=s(\vec{x}\vec{y})}}}{\abs{\Y_{\vec{x}\vec{y}=s(\vec{x}\vec{y})}}} = 1.
    \]
    Then
    \begin{align*}
        \X[F/z](\vec{x}\vec{y}z\mapsto\vec{a}\vec{b}\gamma) &= \X(\vec{x}\vec{y}\mapsto\vec{a}\vec{b})F(\vec{x}\vec{y}\mapsto\vec{a}\vec{b})(\gamma) \\
        &= \abs{\X_{\vec{x}\vec{y}=\vec{a}\vec{b}}}\cdot\frac{\Y(\vec{x}\vec{y}z\mapsto\vec{a}\vec{b}\gamma)}{\abs{\Y_{\vec{x}\vec{y}=\vec{a}\vec{b}}}} \\
        &= \frac{\abs{\X_{\vec{x}=\vec{a}}}}{\abs{\Y_{\vec{x}=\vec{a}}}}\Y(\vec{x}\vec{y}z\mapsto\vec{a}\vec{b}\gamma) \\
        &= \Y(\vec{x}\vec{y}z\mapsto\vec{a}\vec{b}\gamma),
    \end{align*}
    whence $\Y=\X[F/z]$. Thus $\X[F/z]\models\varphi$, so $\X\models\exists z\varphi$.

    Conversely, suppose that $\Y=\X[F/z]$ for some $F$. Then it is clear that $\abs{\X_{\vec{x}=\vec{a}}}=0$ if and only if $\abs{\Y_{\vec{x}=\vec{a}}}=0$ for any $\vec{a}$. Now
    \begin{align*}
        \abs{\Y_{\vec{x}\vec{y}=\vec{a}\vec{b}}}\abs{\X_{\vec{x}=\vec{a}}} &= \abs{\X_{\vec{x}=\vec{a}}}\sum_\gamma\Y(\vec{x}\vec{y}z\mapsto\vec{a}\vec{b}\gamma) \\
        &= \abs{\X_{\vec{x}=\vec{a}}}\sum_\gamma\X(\vec{x}\vec{y}\mapsto\vec{a}\vec{b})F(\vec{x}\vec{y}\mapsto\vec{a}\vec{b})(\gamma) \\
        &= \abs{\X_{\vec{x}=\vec{a}}}\X(\vec{x}\vec{y}\mapsto\vec{a}\vec{b}) \\
        &= \abs{\X_{\vec{x}=\vec{a}}}\abs{\X_{\vec{x}\vec{y}=\vec{a}\vec{b}}} \\
        &= \abs{\X_{\vec{x}\vec{y}=\vec{a}\vec{b}}}\abs{\Y_{\vec{x}=\vec{a}}}.
    \end{align*}
    Thus $\Y$ uniformly realizes $\X$.
\end{proof}

A consequence of Lemma~\ref{419} is that if $\varphi(\vec{x},\vec{y},\vec{z})$ is a formula of probabilistic independence logic and defines a property of probabilistic hidden-variable teams, then $\tilde{\exists} z_0\exists z_1\dots\exists z_{l-1}\varphi$ defines the class of probabilistic empirical teams that are uniformly realized by hidden-variable teams satisfying $\varphi$.

\subsection{$K$-Teams}

In this section we study a version of team semantics using semirings. It can be viewed as a generalization of both ordinary and probabilistic team semantics.

Semiring relations, which we can consider semiring teams, were introduced in~\cite{gkt2007provenance} to study the provenance of relational database queries. In~\cite{gt2017semiring}, provenance of first-order formulas was studied by defining a semiring semantics for first-order logic. A similar approach to team semantics was taken in~\cite{bhkpv2023unified}. $K$-relations are studied in~\cite{hannula2023conditional} as a unifying framework for conditional independence and other similar notions.

In the sheaf-theoretic approach to contextuality and non-locality introduced in \cite{Abramsky_2011}, semiring-valued distributions are used to give a unified account of probabilistic and possibilistic forms of contextuality and non-locality, as well as signed measures (``negative probabilities'').  

\begin{definition}
    A structure $(K, +, \cdot, 0, 1)$ is a (non-trivial) \emph{semiring} if
    \begin{enumerate}
        \item $(K,+,0)$ is a commutative monoid with identity element $0$,
        \item $(K,\cdot,1)$ is a monoid with identity element $1$,
        \item The multiplication (both right and left) distributes over the addition,
        \item $0$ annihilates $K$, i.e. $a\cdot 0 = 0\cdot a = 0$ for all $a\in K$, and
        \item $0\neq 1$.
    \end{enumerate}
\end{definition}
We say that a semiring $K$ is \emph{commutative} if also multiplication is commutative. We say that $K$ is \emph{multiplicatively cancellative} if for all $a,b,c\in K$,
\[
    \text{$ab = ac$ and $a \neq 0$} \implies b = c.
\]
We say that $K$ is \emph{positive} if it is plus-positive, i.e.
\[
    a+b = 0 \implies a=b=0,
\]
and has no zero-divisors, i.e.
\[
    ab = 0 \implies \text{$a = 0$ or $b = 0$}.
\]

Canonical semirings include
\begin{itemize}
    \item the natural numbers $(\N,+,\cdot,0,1)$,
    \item multivariate polynomials $(\N[X_0,\dots,X_{n-1}],+,\cdot,0,1)$ over $\N$,
    \item the Boolean semiring $\BB = (\{0,1\},\lor,\land,0,1)$, and
    \item the non-negative reals $\R_{\geq 0} = ([0,\infty),0,\cdot,0,1)$,
\end{itemize}
as well as all rings.

$K$-teams are the natural generalization of probabilistic teams: all the relevant definitions are the same but with the interval $[0,1]$ of real numbers replaced with $K$.

\begin{definition}
    Let $K$ be a semiring and $X$ the full team of a finite structure $\A$. A $K$-team of $\A$ is a function $\X\colon X\to K$. We denote by $\support\X$ the set $\{s\in X \mid \X(s) \neq 0\}$. By $|\X_{\vu = \va}|$ we denote the sum
    \[
        \sum_{\substack{s(\vec{u})=\vec{a} \\ s\in\support\X}}\X(s).
    \]
\end{definition}

We can view probabilistic teams as $\R_{\geq0}$-teams, multiteams as $\N$-teams and ordinary teams as $\BB$-teams. We say that a $K$-team $\X$ is \emph{total} if $\support\X$ is the full team, i.e. $\X(s) > 0$ for all assignments $s$.

\begin{definition}
    Let $\A$ be a finite structure, $\X$ a $K$-team of $\A$, and $\vu$, $\vv$ and $\vw$ tuples of variables. Then $\A\models_\X \vu\probind_{\vv}\vw$ if for all $\va$, $\vb$ and $\vc$,
    \[
        \abs{\X_{\vec{u}\vec{v}=\vec{a}\vec{b}}} \cdot \abs{\X_{\vec{v}\vec{w}=\vec{b}\vec{c}}} = \abs{\X_{\vec{u}\vec{v}\vec{w}=\vec{a}\vec{b}\vec{c}}} \cdot \abs{\X_{\vec{v}=\vec{b}}}.
    \]
\end{definition}

It is easy to see that the definition of independence is invariant under scaling, i.e. if we denote by $a\X$ the $K$-team $s\mapsto a\X(s)$, then
\[
    \X\models\vu\probind_{\vv}\vw \iff a\X\models\vu\probind_{\vv}\vw.
\]

\begin{proposition}[Hannula~\cite{hannula2023conditional}]
    \label{Proposition (Hannula): Semigraphoid axioms hold very generally}
    The following hold for the semigraphoid axioms:
    \begin{enumerate}
        \item Triviality, Symmetry and Decomposition are sound for all $K$-teams whenever $K$ is a commutative semiring.
        \item If in addition, $K$ is positive and multiplicatively cancellative, also Weak Union and Contraction are sound.
    \end{enumerate}
\end{proposition}

\begin{lemma}
    The reflexivity rule from Definition~\ref{axioms} is sound for all $K$-teams for $K$ a commutative semiring.
\end{lemma}
\begin{proof}
    Let $\X$ be a $K$-team. Let $\va$ and $\vb$ be arbitrary. Then
    \begin{align*}
        \abs{\X_{\vx\vx = \va\va}}\cdot\abs{\X_{\vx\vy = \va\vb}} = \abs{\X_{\vx = \va}}\cdot\abs{\X_{\vx\vx\vy = \va\va\vb}} =\abs{\X_{\vx\vx\vy = \va\va\vb}}\cdot\abs{\X_{\vx = \va}}.
    \end{align*}
    Hence $\X\models\vx\probind_\vx\vy$.
\end{proof}

\begin{corollary}[Probabilistic Soundness Theorem]
    \label{Proposition: Probabilistic soundness theorem}
    If $\varphi$ entails $\psi$ by repeated applications of the rules of Section~\ref{subsec:axioms-of-independence-logic}, then $\varphi\models\psi$ in probabilistic team semantics.
\end{corollary}
\begin{proof}
    \begin{enumerate}
        \item Soundness of the axioms of the independence atom follows from Proposition~\ref{Proposition: Axioms of independence atom follow from semigraphoid axioms}.
        \item \emph{Dependence introduction:} Follows from Corollary~\ref{Corollary: Probabilistic dependence logic is weakly flat and strongly downwards closed}.
        \item \emph{Elimination of existential quantifier:}
    Follows from Lemma~\ref{Lemma: Locality of probabilistic independence logic} and the observation that if $x\notin V$, then $\X[F/x]\restriction V = \X\restriction V$ for any function $F$.
        \item \emph{Introduction of existential quantifier:}
        Suppose that $y$ does not occur in the range of $\exists x$ or $\forall x$ in $\varphi$. Suppose that $\X\models\varphi(y/x)$. Then define a function $F$ by setting
        \[
            F(s)(a) =
            \begin{cases}
                1 & \text{if $s(y)=a$,} \\
                0 & \text{otherwise.}
            \end{cases}
        \]
        Then clearly $\X[F/x]$ is the same distribution on assignments $s(s(y)/x)$ as $\X$ is on assignments $s$. Thus $\X[F/x]\models\varphi$ and hence $\X\models\exists x\varphi$.
    \end{enumerate}
\end{proof}

\subsection{Properties of Probabilistic Teams}\label{subsec:properties-of-probabilistic-teams}

By simply replacing the symbol $\perp$ by the symbol $\probind$, we get the probabilistic versions of the previously introduced possibilistic team properties of empirical and hidden-variable teams. These are in line with the definitions in~\cite{MR2443068}, with the exception of no-signalling and parameter independence which suffer from the same weakness as their possibilistic counterparts and which we have generalized here.

\begin{definition}[Probabilistic Team Properties]
    \quad
    \label{Definition: Probabilistic Team Properties}
    \begin{enumerate}
        \item A probabilistic empirical team $\X$ supports \emph{probabilistic no-signalling} if it satisfies the formula
        \begin{equation}
            \label{eq: Probabilistic No-Signalling}
            \bigwedge_{I\subseteq n}\{x_i \mid i\notin I\}\probind_{\{x_i \mid i\in I\}\vec z}\{y_i \mid i\in I\}. \tag{PNS}
        \end{equation}

        \item A probabilistic hidden-variable team $\X$ supports \emph{probabilistic weak determinism} if it satisfies the formula
        \begin{equation}
            \label{eq: Probabilistic Weak Determinism}
            \bigwedge_{i<n} \=({\vec{x}\vec{z}},y_i). \tag{PWD}
        \end{equation}

        \item A probabilistic hidden-variable team $\X$ supports \emph{probabilistic strong determinism} if it satisfies the formula
        \begin{equation}
            \label{eq: Probabilistic Strong Determinism}
            \bigwedge_{i<n} \=({x_i\vec{z}},y_i). \tag{PSD}
        \end{equation}

        \item A probabilistic hidden-variable team $\X$ supports \emph{probabilistic single-valuedness} if it satisfies the formula
        \begin{equation}
            \label{eq: Probabilistic Single-Valuedness}
            \=(\vec{z}). \tag{PSV}
        \end{equation}

        \item A probabilistic hidden-variable team $\X$ supports \emph{probabilistic $\vz$-inde\-pend\-ence} if it satisfies the formula
        \begin{equation}
            \label{eq: Probabilistic z-Independence}
            \vec{z}\probind\vec{x}. \tag{P$z$I}
        \end{equation}

        \item A probabilistic hidden-variable team $\X$ supports \emph{probabilistic parameter independence} if it satisfies the formula
        \begin{equation}
            \label{eq: Probabilistic Parameter Independence}
            \bigwedge_{I\subseteq n}\{x_i \mid i\notin I\}\probind_{\{x_i \mid i\in I\}\vec z}\{y_i \mid i\in I\}. \tag{PPI}
        \end{equation}

        \item  A probabilistic hidden-variable team $\X$ supports \emph{probabilistic outcome independence} if it satisfies the formula
        \begin{equation}
            \label{eq: Probabilistic Outcome Independence}
            \bigwedge_{i<n} y_i\probind_{\vec{x}\vec{z}}\{y_j \mid j\neq i\}. \tag{POI}
        \end{equation}
    \end{enumerate}

    As we did not have a syntactic formula for locality, we need to give an explicit semantic definition for probabilistic locality as well.

    \begin{enumerate}[resume]
        \item A probabilistic hidden-variable team $\X$ supports \emph{probabilistic locality} if for all $\vec{a}$, $\vec{b}$ and $\vec{\gamma}$ we have
        \[
            \abs{\X_{\vec{x}\vec{y}\vec{z}=\vec{a}\vec{b}\vec{\gamma}}}\prod_{i<n}\abs{\X_{x_i\vec{z} = a_i \vec{\gamma}}} = \abs{\X_{\vec{x}\vec{z}=\vec{a}\vec{\gamma}}}\prod_{i<n}\abs{\X_{x_i y_i \vec{z} = a_i b_i \vec{\gamma}}}.
        \]
    \end{enumerate}
\end{definition}

Lemma~\ref{Lemma: Locality is equivalent to parameter independence + outcome independence}, stating that locality is equivalent to the conjunction of parameter and outcome independence, remains true in the probabilistic world, at least with the simpler definition of parameter independence from~\cite{MR2443068}.

\begin{lemma}
    Probabilistic locality is equivalent to the formula
    \[
        \bigwedge_{i<n}\left(\{ x_j \mid j\neq i \}\probind_{x_i\vec{z}}y_i \land y_i\probind_{\vec{x}\vec{z}}\{y_j \mid j\neq i\}\right).
    \]
\end{lemma}
\begin{proof}
    Essentially proved in~\cite{MR2443068}.
\end{proof}
We conjecture that even with the more general definition of parameter independence, probabilistic locality is equivalent to the conjunction of probabilistic parameter independence and probabilistic outcome independence, as it is in the possibilistic case.

\begin{corollary}
    \label{Corollary: Probabilistic implies possibilistic}
    For any of the properties is Definition~\ref{Definition: Probabilistic Team Properties}, if a probabilistic team supports it, then the possibilistic collapse supports the corresponding possibilistic property.
\end{corollary}
\begin{proof}
    An immediate consequence of Proposition~\ref{Proposition: Truth in independence logic is preserved by possibilistic collapse}.
\end{proof}

As by Proposition~\ref{Proposition: Probabilistic soundness theorem} the axioms presented in Section~\ref{subsec:axioms-of-independence-logic} are valid also in the probabilistic setting, all the results from Section~\ref{subsec:relationships-between-the-properties} that were proved from the axioms are true also for probabilistic teams.

\begin{corollary}
    The following hold for probabilistic teams (and more generally for $K$-teams whenever $K$ is a commutative, positive and multiplicatively cancellative semiring):
    \begin{enumerate}
        \item\label{corollary first item} $\=(\vec{x}\vec{z},\vec{y})\models\bigwedge_{i<n}y_i\probind_{\vec{x}\vec{z}}\{y_j \mid j\neq i\}$,
        \item $\=(x_i\vec{z},y_i)\models\{x_j \mid j\neq i\}\probind_{x_i\vec{z}}y_i$,
        \item $\bigwedge_{i<n}\{x_j \mid j\neq i\}\probind_{x_i\vec{z}}y_i\land\=(\vec{x}\vec{z},\vec{y})\ \models\ \bigwedge\=(x_i\vec{z},y_i)$,
        \item $\varphi\models\exists z_0\dots\exists z_{l-1}(\=(\vec{z})\land\varphi)$.
        \item The following formulas are equivalent:
        \begin{enumerate}
            \item $\bigwedge_{i<n}\{x_j \mid j\neq i\}\probind_{x_i} y_i$,
            \item $\tilde\exists z_0\exists z_1\dots\exists z_{l-1} \left( \vec{z}\probind\vec{x} \land \bigwedge_{i<n} \{ x_j \mid j\neq i \}\probind{x_i\vec{z}}y_i \right)$.
        \end{enumerate}
        \item\label{corollary last item} The following formulas are equivalent:
        \begin{enumerate}
            \item $\bigwedge_{I\subseteq n}\{x_i \mid i\notin I\}\probind_{\{x_i\mid i\in I\}}\{y_i \mid i\in I\}$,
            \item $\tilde{\exists}z_0\exists z_1\dots\exists z_{l-1} \left( \vec{z}\probind\vec{x} \land \bigwedge_{I\subseteq n} \{x_i \mid i\notin I\}\probind_{\{x_i\mid i\in I\}\vec{z}}\{y_i\mid i\in I\} \right)$.
        \end{enumerate}
    \end{enumerate}
\end{corollary}

The above \ref{corollary first item}-\ref{corollary last item} may seem like somewhat arbitrary observations. However, let us recall that they arise from examples motivated by quantum mechanics and each one of them has an intuitive interpretation in physics. It would seem more satisfactory to present a systematic study of such logical consequences and equivalences but we have already observed that it would be a formidable task bordering the impossible.

\subsection{Building Probabilistic Teams}\label{subsec:building-probabilistic-teams}

As we saw in section~\ref{subsec:properties-of-probabilistic-teams}, properties of probabilistic teams are inherited by their possibilistic collapses. Here we prove results concerning the question to what extent the converse holds: when can one construct a probabilistic team out of a possibilistic one, with the same properties?

Following~\cite{MR3038040}, we proceed to show that some no-signalling teams have no probabilistic realization that would also support probabilistic no-signalling.

This also shows that $\varphi\models\PR\varphi$ is not true for all formulas $\varphi$ of independence logic.

\begin{proposition}
    \label{Proposition: Possibilistic no-signalling does not imply probabilistic no-signalling}
    Suppose that $n=2$. There is an empirical team $X$ supporting no-signalling such that there is no probabilistic team $\X$ that supports probabilistic no-signalling and whose possibilistic collapse is $X$, i.e.
    \[
        \bigwedge_{i<n}\{x_j \mid j\neq i\}\perp_{x_i}y_i\ \not\models\ \PR\bigwedge_{i<n}\{x_j \mid j\neq i\}\perp_{x_i}y_i.
    \]
\end{proposition}
\begin{proof}
    We let $X = \{s_0,\dots,s_{11}\}$, where the assignments $s_i$ are as follows:
    \begin{center}
        \begin{tabular}{|c|cccc|}
            \hline
            & $x_0$ & $x_1$ & $y_0$ & $y_1$ \\
            \hline
            $s_0$ & $0$   & $0$   & $0$   & $0$   \\
            $s_1$ & $0$   & $0$   & $0$   & $1$   \\
            $s_2$ & $0$   & $0$   & $1$   & $1$   \\
            $s_3$ & $0$   & $1$   & $0$   & $0$   \\
            $s_4$ & $0$   & $1$   & $1$   & $0$   \\
            $s_5$ & $0$   & $1$   & $1$   & $1$   \\
            \hline
        \end{tabular}
        \begin{tabular}{|c|cccc|}
            \hline
            & $x_0$ & $x_1$ & $y_0$ & $y_1$ \\
            \hline
            $s_6$    & $1$   & $0$   & $0$   & $0$   \\
            $s_7$    & $1$   & $0$   & $1$   & $0$   \\
            $s_8$    & $1$   & $0$   & $1$   & $1$   \\
            $s_9$    & $1$   & $1$   & $0$   & $0$   \\
            $s_{10}$ & $1$   & $1$   & $0$   & $1$   \\
            $s_{11}$ & $1$   & $1$   & $1$   & $1$   \\
            \hline
        \end{tabular}
    \end{center}
    It is straightforward to check that $X$ supports no-signalling. Suppose for a contradiction that $\X$ is a probabilistic team that supports probabilistic no-signalling and whose possibilistic collapse is $X$. Then there are positive numbers $p_0,\dots,p_{11}$ with $\sum_{i<12}p_i = 1$ such that $\X(s_i)=p_i$ for all $i<12$. By probabilistic no-signalling,
    \[
        \X\models x_{1-i}\probind_{x_i} y_i,
    \]
    for all $i\in\{0,1\}$, so we have, for all $a,b,c,i\in\{0,1\}$,
    \[
        \abs{\X_{x_{1-i} x_i = ab}}\cdot\abs{\X_{x_i y_i = bc}} = \abs{\X_{x_{1-i} x_i y_i = abc}}\cdot\abs{\X_{x_i=b}}.
    \]
    Calculating the marginal probabilities and applying the above condition, we get the following four equations:
    \begin{enumerate}
        \item $p_2 p_3 = (p_0 + p_1)(p_4 + p_5)$,
        \item $p_0 p_8 = (p_1 + p_2)(p_6 + p_7)$,
        \item $p_6 p_{11} = (p_7 + p_8)(p_9 + p_{10})$, and
        \item $p_5 p_9 = (p_3 + p_4)(p_{10} + p_{11})$.
    \end{enumerate}
    From this, using the third and the fourth equations, we get
    \[
        p_6 p_{11} p_5 p_9 = (p_7 + p_8)(p_9 + p_{10})(p_3 + p_4)(p_{10} + p_{11}) > p_8 p_9 p_3 p_{11},
    \]
    whence $p_5 p_6 > p_3 p_8$. Then by multiplying by $p_2$ and using the first equation, we get
    \[
        p_2 p_5 p_6 > p_2 p_3 p_8 = (p_0 + p_1)(p_4 + p_5)p_8 > p_0 p_5 p_8,
    \]
    whence $p_2 p_6 > p_0 p_8$. Then finally, using the second equation, we get
    \[
        p_2 p_6 > p_0 p_8 = (p_1 + p_2)(p_6 + p_7) > p_2 p_6,
    \]
    which is a contradiction.
\end{proof}

A minimal example of a possibilistic no-signalling team which is not a collapse of any probabilistic no-signalling team can be obtained by translating an example of~\cite{MR3552129}---which occurs as part of a discussion about the question whether there exists an intrinsic characterization of the class of no-signalling teams that are collapses of probabilistic no-signalling teams---into our team-semantic framework.

The next property of empirical and hidden-variable teams, \emph{measurement locality}, was introduced by \samson in~\cite{MR3038040}. Measurement locality states that the measurement variables are mutually independent of each other.

\begin{definition}[Measurement Locality]
    An empirical team $X$ supports \emph{measurement locality} if it satisfies the formula
    \begin{equation}
        \label{eq: Measurement Locality}
        \bigwedge_{i<n} x_i\perp\{x_j \mid j\neq i\}. \tag{ML}
    \end{equation}

    A hidden-variable team $X$ supports measurement locality if it satisfies the formula
    \begin{equation}
        \bigwedge_{i<n} x_i\perp_{\vec{z}}\{x_j \mid j\neq i\}. \tag{ML}
    \end{equation}
\end{definition}

\begin{definition}[Probabilistic Measurement Locality]
    A probabilistic empirical team $\X$ supports \emph{probabilistic measurement locality} if it satisfies the formula
    \begin{equation}
        \label{eq: Probabilistic Measurement Locality}
        \bigwedge_{i<n} x_i\probind\{x_j \mid j\neq i\}. \tag{PML}
    \end{equation}

    A probabilistic hidden-variable team $\X$ supports probabilistic measurement locality if it satisfies the formula
    \begin{equation}
        \bigwedge_{i<n} x_i\probind_{\vec{z}}\{x_j \mid j\neq i\}. \tag{PML}
    \end{equation}
\end{definition}

\begin{corollary}
    Whenever a probabilistic team supports probabilistic measurement locality, the possibilistic collapse supports measurement locality.
\end{corollary}
\begin{proof}
    An immediate consequence of Proposition~\ref{Proposition: Truth in independence logic is preserved by possibilistic collapse}.
\end{proof}

\begin{definition}
    Given sets $A=\prod_{i<n}A_i$ and $B=\prod_{i<n}B_i$ and a probability distribution $p_{\vec{a}}$ on $B$ for each $\vec{a}\in A$, we say that a probabilistic empirical team $\X$ is a \emph{uniform joint distribution} of the outcome distribution family $\{p_{\vec{a}} \mid \vec{a}\in A\}$ if the value domain of $\X$ is $\bigcup_{i<n}(A_i\cup B_i)$ and $\X(s) = p_{s(\vec{x})}(s(\vec{y}))/\abs{A}$ whenever $s(x_i)\in A_i$ and $s(y_i)\in B_i$ for all $i<n$, and $\X(s)=0$ otherwise.

    Similarly, given a set $\Gamma$ of possible hidden-variable values and outcome distributions $p_{\vec{a}\vec{\gamma}}$ on $B$ for $\vec{a}\in A$ and $\vec{\gamma}\in\Gamma$, we say that a probabilistic hidden-variable team $\X$ is a uniform joint distribution of the outcome distribution family if $X(s)=p_{s(\vec{x}\vec{z})}(s(\vec{y}))/\abs{A\times\Gamma}$.
\end{definition}

\begin{proposition}
    \label{Lemma: Measurement locality is automatically supported by certain teams}
    A uniform joint distribution of an outcome distribution family supports probabilistic measurement locality.
\end{proposition}
\begin{proof}
    First observe that
    \begin{align*}
        \abs{\X_{x_i\vec{z} = a_i\vec{\gamma}}} &= \sum_{\substack{\vec{c}\in A \\ c_i=a_i}}\sum_{\vec{b}\in B}\frac{p_{\vec{c}\vec{\gamma}}(\vec{b})}{\abs{A\times\Gamma}} = \sum_{\substack{\vec{c}\in A \\ c_i=a_i}}\frac{1}{\abs{A}\abs{\Gamma}} \\
        &= \frac{1}{\abs{A}\abs{\Gamma}}\abs{\{\vec{c}\in A \mid c_i=a_i\}} \\
        &= \frac{1}{\abs{\Gamma}\prod_{j<n}\abs{A_j}}\prod_{\substack{j<n \\ j\neq i}}\abs{A_j} \\
        &= \frac{1}{\abs{\Gamma}\abs{A_i}}.
    \end{align*}
    Then we have
    \begin{align*}
        \abs{\X_{\vec{x}\vec{z} = \vec{a}\vec{\gamma}}}\cdot\abs{\X_{\vec{z}=\vec{\gamma}}}^{n-1} &= \left( \sum_{\vec{b}\in B} \frac{p_{\vec{a}\vec{\gamma}}(\vec{b})}{\abs{A\times\Gamma}} \right)\left( \sum_{\vec{c}\in A}\sum_{\vec{b}\in B}\frac{p_{\vec{c}\vec{\gamma}}(\vec{b})}{\abs{A\times\Gamma}} \right)^{n-1} \\
        &= \frac{1}{\abs{A}\abs{\Gamma}}\left( \sum_{\vec{c}\in A}\frac{p_{\vec{c}\vec{\gamma}}(\vec{b})}{\abs{A}\abs{\Gamma}} \right)^{n-1} = \frac{1}{\abs{A}\abs{\Gamma}}\cdot\left( \frac{\abs{A}}{\abs{A}\abs{\Gamma}} \right)^{n-1} \\
        &= \frac{1}{\abs{A}\abs{\Gamma}^n} = \frac{1}{\abs{\Gamma}^n}\prod_{i<n}\frac{1}{\abs{A_i}} = \prod_{i<n}\frac{1}{\abs{\Gamma}\abs{A_i}} \\
        &= \prod_{i<n}\abs{\X_{x_i\vz = a_i\vec\gamma}}.
    \end{align*}
    Now we observe the following fact that is easy to prove by induction on $n$:
    a~probabilistic team $\Y$ satisfies the formula $\bigwedge_{i<n} v_i\probind_{\vec{u}}\{v_j \mid j\neq i\}$
    if and only if for all $\vec{a}$ and $\vec{b}$,
    \[
        \abs{\Y_{\vec{v}\vec{u}=\vec{a}\vec{b}}}\cdot\abs{\Y_{\vec{u}=\vec{b}}}^{n-1} = \prod_{i<n}\abs{\Y_{v_i \vec{u} = a_i \vec{b}}}.
    \]
    Each $v_i$ can also be replaced by a tuple of variables.
    From this and the above calculations, it follows that $\X\models x_i\probind_{\vec{z}}\{x_j \mid j\neq i\}$.
\end{proof}

Next we show that there is a canonical way of constructing a probabilistic team out of a possibilistic hidden-variable team that supports $\vz$-independence, and that such a probabilistic team will support locality, measurement locality and $\vz$-independence if its possibilistic collapse does.

\begin{definition}
    Given a hidden-variable team $X$ that supports $\vz$-independence, we define the probabilistic hidden-variable team $\probabilisticLift{X}$ as follows. Denote
    \begin{align*}
        \Gamma &= \{ s(\vec{z}) \mid s\in X \}, \\
        M &= \{ s(\vec{x}) \mid s\in X \}, \\
        O_{\vec{a},\vec{\gamma}} &= \{ s(\vec{y}) \mid s\in X, s(\vec{x}\vec{z}) = \vec{a}\vec{\gamma} \},
    \end{align*}
    $\mh = \abs{\Gamma}$, $\mm = \abs{M}$, and $\mo(\vec{a},\vec{\gamma}) = \abs{O_{\vec{a},\vec{\gamma}}}$. We then define $\probabilisticLift{X}$ by setting
    \[
        \probabilisticLift{X}(s) =
        \begin{cases}
            \displaystyle\frac{1}{\mh\cdot\mm\cdot\mo(s(\vec{x}),s(\vec{z}))} & \text{if $s(\vec{x})\in M$ and $s(\vec{z})\in\Gamma$}, \\
            0 & \text{otherwise}.
        \end{cases}
    \]
\end{definition}

\begin{lemma}
    \label{Lemma: Probabilistic lift is well-defined}
    $\probabilisticLift{X}$ is well-defined and $\support\probabilisticLift{X} = X$.
\end{lemma}
\begin{proof}
    First, as $X$ supports $\vz$-independence, for every $s,s'\in X$ we can find $s''\in X$ with $s''(\vec{x})=s(\vec{x})$ and $s''(\vec{z})=s'(\vec{z})$, and thus, given an assignment $s$, the condition
    \[
        s(\vec{x})\in M \text{ and } s(\vec{z})\in\Gamma
    \]
    implies that there is some $s'\in X$ with $s'(\vec{x}\vec{z})=s(\vec{x}\vec{z})$ and thus the number $\mo(s(\vec{x}),s(\vec{z}))$ is non-zero. Hence $\probabilisticLift{X}$ is well-defined as a function. What is left to show is that $\probabilisticLift{X}$ is a probability distribution. Let us notice that for each $\vec{a}\vec{b}\vec{\gamma}$, the probability of the assignment $\vec{x}\vec{y}\vec{z}\mapsto\vec{a}\vec{b}\vec{\gamma}$ does not depend on $\vec{b}$, so each assignment $s$ with $s(\vec{x}\vec{z})=\vec{a}\vec{\gamma}$ has an equal probability, which is $1/(\mh\mm\mo(\vec{a},\vec{\gamma}))$, and thus the joint probability of such assignments is
    \begin{align*}
        \abs{\probabilisticLift{X}_{\vec{x}\vec{z}=\vec{a}\vec{\gamma}}} &= \sum_{\vec{b}}\probabilisticLift{X}(\vec{x}\vec{y}\vec{z}\mapsto\vec{a}\vec{b}\vec{\gamma}) = \frac{\mo(\vec{a},\vec{\gamma})}{\mh\mm\mo(\vec{a},\vec{\gamma})} = \frac{1}{\mh\mm}.
    \end{align*}
    This, in turn, does not depend on $\vec{a}$ or $\vec{\gamma}$. Also, by $\vz$-independence we have $\abs{\{s \mid s(\vec{x}\vec{z}) = \vec{a}\vec{\gamma} \}} = \mm\cdot\mh$. Thus
    \begin{align*}
        \sum_{s\in X} \probabilisticLift{X}(s) &= \sum_{\vec{a}\vec{b}\vec{\gamma}} \probabilisticLift{X}(\vec{x}\vec{y}\vec{z} \mapsto \vec{a}\vec{b}\vec{\gamma}) = \sum_{\vec{a}\vec{\gamma}} \frac{1}{\mh\mm} = \mm\mh\cdot\frac{1}{\mh\mm} = 1.
    \end{align*}
    Thus $\probabilisticLift{X}$ is a well-defined distribution. Clearly the collapse of $\probabilisticLift{X}$ is $X$.
\end{proof}

In contrast to Proposition~\ref{Proposition: Possibilistic no-signalling does not imply probabilistic no-signalling}, we now obtain:

\begin{proposition}
    \label{Proposition: Measurement locality + z-independence + locality implies the probabilistic versions}
    Let $X$ be a hidden-variable team supporting measurement locality, $\vz$-independence and locality. Then $\probabilisticLift{X}$ supports probabilistic measurement locality, probabilistic $\vz$-independence and probabilistic locality whose possibilistic collapse is $X$. Thus, the formula
    \[
        \varphi \coloneqq \vec{z}\perp\vec{x}\land \bigwedge_{i<n}x_i y_i\perp_{\vec{z}}\{x_j y_j \mid j\neq i\}
    \]
    satisfies $\varphi\models\PR\varphi$.
\end{proposition}
\begin{proof}
    Essentially proved in~\cite{MR3038040}.
\end{proof}

\section{Empirical Teams Arising from Quantum Mechanics}\label{sec:empirical-teams-arising-from-quantum-mechanics}

A team, even what we call an empirical team,  is in itself just an abstract set of assignments.
It does not need to have any ``provenance'', although
in practical applications teams arise from concrete data. In our current context of quantum mechanics,
we use the abstract concept of a team for implications which indeed are totally general and abstract.
However, when it comes to counter-examples demonstrating that some implications are not valid,
the question arises whether our example teams are ``merely'' abstract or whether they can actually
arise in experiments. One of the beauties of quantum physics is that we have a precise mathematical axiomatization of quantum mechanics, essentially due to von Neumann \cite{von2018mathematical}.
This axiomatization is formulated in terms of operators on complex Hilbert spaces. We shall limit our discussion to the finite-dimensional case, where operators can be represented as complex matrices.

\begin{definition}
    Let $M$ and $O$ be sets of $n$-tuples (the ``set of measurements'' and the ``set of outcomes''), and, for $i<n$, denote $M_i = \{a_i \mid \vec{a}\in M \}$ and $O_i = \{ b_i \mid \vec{b}\in O \}$. A \emph{finite-dimensional tensor-product quantum system} of type $(M,O)$ is a tuple
    \[
        \system \coloneqq (\hilbert,(A_i^{a,b})_{a\in M_i, b\in O_i, i<n},\rho),
    \]
    where
    \begin{itemize}
        \item $\hilbert$ is the tensor product $\bigtensor_{i<n}\hilbert_i$ of finite-dimensional complex Hilbert spaces $\hilbert_i$, $i<n$,
        \item for all $i<n$ and $a\in M_i$, $\{A_i^{a,b} \mid b\in O_i \}$ is a positive operator-valued measure (POVM)\footnote{I.e. $A_i^{a,b}$ are positive-definite operators such that $\sum_{b\in O_i}A_i^{a,b} = \id_{\hilbert_i}$ for all $a\in M_i$.} on $\hilbert_i$, and
        \item $\rho$ is a density operator on $\hilbert$ (the ``state of $\system$''), i.e.
        \[
            \rho = \sum_{j<k}p_j\ket{\psi_j}\bra{\psi_j},
        \]
        where $\ket{\psi_j}\in \hilbert$ and $p_j\in[0,1]$ for all $j<k$ and $\sum_{j<k}p_j = 1$.
    \end{itemize}
    For each measurement $\vec{a}\in M$, we define the probability distribution $p^\system_{\vec{a}}$ of outcomes by setting $p^\system_{\vec{a}}(\vec{b}) \coloneqq \trace(\bigtensor_{i<n} A_i^{a_i,b_i}\rho)$, where $\trace(L)$ denotes the trace of the matrix $L$.
\end{definition}

\begin{definition}\label{5.2}
    Let $\X$ be a probabilistic team with variable domain $\Vm\cup\Vo$. Denote $M = \{s(\vec{x}) \mid s\in\support\X \}$ and $O = \{s(\vec{y}) \mid s\in\support\X \}$.
       
        We say that $\X$ is a \emph{finite-dimensional tensor-product quantum-mechanical team} if there exists a finite-dimensional tensor-product quantum system $\system$ of type $(M,O)$ such that for all assignments $s$, we have
        \[
            \X(s)=p^\system_{s(\vec{x})}(s(\vec{y}))/\abs{M}.
        \]
        
    We call an empirical probabilistic team $\X$ a \emph{finite-dimensional tensor-product quantum-mechanical realization} of an empirical possibilistic team $X$ if $X$ is the possibilistic collapse of $\X$ and $\X$ is finite-dimensional tensor-product quantum-mechanical.
\end{definition}

Denote by $\QT$ the set of finite-dimensional tensor-product quantum-mechani\-cally realizable teams.

We can define a new atomic formula $\QR$ such that $X\models\QR$ if $X$ has a finite-dimensional tensor-product quantum realization. In other words, $X\models\QR$ if and only if $X\in\QT$. Then one can ask what kind of properties this atom has. More generally, we can define an operation $\QR$ by
\begin{center}
    $X\models\QR\varphi$ if $X$ has a finite-dimensional tensor-product quantum-mechanical realization $\X$ such that $\X\models\varphi$,
\end{center}
analogously to the operation $\PR$.

One can also ask what kind of property of probabilistic teams being finite-dimensional tensor-product quantum-mechanical is. In~\cite{MR3802381}, Durand et al. showed that probabilistic independence logic (with rational probabilities) is equivalent to a probabilistic variant of existential second-order logic $\mathrm{ESOf}_{\Q}$. Is being finite-dimensional tensor-product quantum-mechanical expressible in $\mathrm{ESOf}_{\R}$ or do we need more expressivity?

We now observe that the set $\QT = \{X \mid X\models\QR\}$ is undecidable but recursively enumerable. For this purpose, we briefly introduce non-local games and then apply a result of Slofstra~\cite{MR3898717}.

\newcommand{\Alice}{\mathrm{A}}
\newcommand{\Bob}{\mathrm{B}}

\begin{definition}\quad
    \begin{enumerate}
        \item Let $I_\Alice$, $I_\Bob$, $O_\Alice$ and $O_\Bob$ be finite sets and let $V\colon O_\Alice\times O_\Bob\times I_\Alice\times I_\Bob\to\{0,1\}$ be a function\footnote{We write $V(a,b\mid c,d)$ for the function value.}. A \emph{(two-player one-round) non-local game} $G$ with question sets $I_\Alice$ and $I_\Bob$, answer sets $O_\Alice$ and $O_\Bob$ and decision predicate $V$ is defined as follows: the first player (Alice) receives an element $c\in I_\Alice$ and the second player (Bob) receives an element $d\in I_\Bob$. Alice returns an element $a\in O_\Alice$ and Bob returns an element $b\in O_\Bob$. The players are not allowed to communicate the received inputs or their chosen outcomes to each other. The players win if $V(a,b\mid c,d)=1$ and lose otherwise.

        \item Let $G$ be a non-local game. A strategy for $G$ is a function $p\colon O_\Alice\times O_\Bob\times I_\Alice\times I_\Bob\to[0,1]$ such that for each pair $(c,d)\in I_\Alice\times I_\Bob$ the function $(a,b)\mapsto p(a,b\mid c,d)$ is a probability distribution. A strategy $p$ is \emph{perfect} if $V(a,b\mid c,d)=0$ implies $p(a,b\mid c,d)=0$.

        \item Let $G$ be a non-local game and $p$ a strategy for $G$. We say that $p$ is a \emph{quantum strategy} if there are finite-dimensional Hilbert spaces $H_\Alice$ and $H_\Bob$, a quantum state $\rho$ of $H_\Alice\tensor H_\Bob$, a POVM $(M^c_a)_{a\in O_\Alice}$ on $H_\Alice$ for each $c\in I_\Alice$ and a POVM $(N^d_b)_{b\in O_\Bob}$ on $H_\Bob$ for each $d\in I_\Bob$ such that
        \[
            p(a,b\mid c,d) = \trace(M^c_a\tensor N^d_b\rho)
        \]
        for all $(a,b,c,d)\in O_\Alice\times O_\Bob\times I_\Alice\times I_\Bob$.
    \end{enumerate}
\end{definition}

\begin{theorem}[Slofstra~\cite{MR3898717}]
    \label{Theorem: It is undecidable to determine if a non-local game has a perfect quantum strategy}
    It is undecidable to determine whether a non-local game has a perfect quantum strategy.
\end{theorem}

\begin{proposition}
    \label{Proposition: Non-local games with perfect quantum strategy reduce to teams with quantum realization}
    There is a many-one reduction from non-local games that have a perfect quantum strategy to teams that have a finite-dimensional tensor-product quantum-mechanical realization.
\end{proposition}
\begin{proof}
    Let $G$ be a game with question sets $I_\Alice$ and $I_\Bob$ and answer sets $O_\Alice$ and $O_\Bob$ and decision predicate $V$. We may assume that for each $c\in I_\Alice$ and $d\in I_\Bob$ there are some $a\in O_\Alice$ and $b\in O_\Bob$ such that $V(a,b \mid c,d)=1$, otherwise we may just map $G$ into the empty team. We let $X_G$ be the set of all assignments $s$ with domain $\{x_0,x_1,y_0,y_1\}$ such that $s(x_0)\in I_\Alice$, $s(x_1)\in I_\Bob$, $s(y_0)\in O_\Alice$, $s(y_1)\in O_\Bob$ and $V(s(y_0),s(y_1) \mid s(x_0),s(x_1))=1$. Let $M = I_\Alice\times I_\Bob$ and
    \begin{align*}
        O &= \{(a,b)\in O_\Alice\times O_\Bob \mid \text{$V(a,b \mid c,d)=1$ for some $c\in I_\Alice$ and $d\in I_\Bob$} \},
    \end{align*}
    and denote by $M_0$, $M_1$, $O_0$ and $O_1$ the appropriate projections of $M$ and $O$. Then $M=\{s(x_0x_1) \mid s\in X_G\}$ and $O=\{s(y_0y_1) \mid s\in X_G\}$.

    We show that $G$ has a perfect quantum strategy if and only if $X_G$ is realizable by a finite-dimensional tensor-product quantum-mechanical team. We only show one direction, the other is similar.
    Suppose that $p$ is a perfect quantum strategy for $G$. Then there are finite-dimensional Hilbert spaces $H_\Alice$ and $H_\Bob$, a quantum state $\rho$ of $H_\Alice\tensor H_\Bob$, a POVM $\{ M^c_a \mid a\in O_\Alice \}$ on $H_\Alice$ for each $c\in I_\Alice$ and a POVM $\{ N^d_b \mid b\in O_\Bob \}$ on $H_\Bob$ for each $d\in I_\Bob$ such that
    \[
        p(a,b\mid c,d)=\trace(M^c_a\tensor N^d_b\rho)
    \]
    for all $(a,b,c,d)\in O_\Alice\times O_\Bob\times I_\Alice\times I_\Bob$. We now define a quantum system
    \[
        \mathcal{S}=(\hilbert, (A^{c,a}_i)_{c\in M_i, a\in O_i, i<2},\rho)
    \]
    of type $(M,O)$ by setting
    \begin{itemize}
        \item $\hilbert = \hilbert_\Alice\tensor\hilbert_\Bob$, and
        \item $A^{c,a}_0 = M^c_a$ and $A^{d,b}_1 = N^d_b$ for all $a\in O_0$, $b\in O_1$, $c\in M_0$ and $d\in M_1$.
    \end{itemize}
    Now clearly
    \begin{align*}
        p^\mathcal{S}_{(c,d)}(a,b) &= \trace(A^{(c,d),(a,b)}\rho) = \trace(M^c_a\tensor N^d_b\rho) = p(a,b \mid c,d).
    \end{align*}
    As $p$ is a perfect strategy, we have $p(a,b \mid c,d)=0$ for any $a$,$b$,$c$ and $d$ such that $V(a,b \mid c,d)=0$. Thus the probabilistic team $\X$ arising from the quantum system $\mathcal{S}$ is such that $\X(s)>0$ if and only if $V(s(y_0),s(y_1) \mid s(x_0), s(x_1))=1$. Hence the possibilistic collapse of $\X$ is $X_G$, and thus $\X$ is a finite-dimensional tensor-product quantum-mechanical realization of $X_G$.
\end{proof}

\begin{corollary}
    The set $\{X \mid X\models\QR\}$ is undecidable but recursively enumerable.
\end{corollary}
\begin{proof}
    Undecidability follows from Theorem~\ref{Theorem: It is undecidable to determine if a non-local game has a perfect quantum strategy} and Proposition~\ref{Proposition: Non-local games with perfect quantum strategy reduce to teams with quantum realization}.
    It is not difficult to show that the problem of determining whether a team has a probabilistic realization which corresponds to a quantum system of dimension $d$ is reducible to the existential theory of the reals, which is known to be in PSPACE~\cite{10.1145/62212.62257}. Hence one can check for each dimension $d$ whether a team has a quantum realization of dimension $d$, and thus we obtain an r.e.~algorithm.
\end{proof}

It is also possible to define wider notions of quantum realizability by dropping the  finite-dimensionality requirement.\footnote{Once finite-dimensionality is dropped, it is also relevant to replace the assumption of tensor product structure by weaker commuting operator assumptions \cite{MR3898717}.} One can then leverage the results in \cite{coladangelo2018unconditional,MR3898717,ji2022mipre}
to show that these lead to strictly larger classes of teams.

The teams we used in Section~\ref{subsec:relationships-between-the-properties} to prove the no-go theorems of quantum mechanics are all quantum realizable.
The following are essentially proved in~\cite{MR3038040}.

\begin{proposition}\quad
    \label{Proposition: No-go theorems come from quantum mechanics}
    \begin{enumerate}
        \item There is a finite-dimensional tensor-product quantum-mechanical team that realizes a GHZ team.

        \item There is a finite-dimensional tensor-product quantum-mechanical team that realizes a Hardy team.
    \end{enumerate}
\end{proposition}

\begin{corollary}\quad
    \label{Corollary: Probabilistic EPR}
        There is a finite-dimensional tensor-product quantum-mechanical team which is not realized by any probabilistic hidden-variable team supporting probabilistic $\vz$-independ\-ence and probabilistic locality; hence
        \[
            \QR \not\models \exists\vec{z} \left( \vec{z}\probind\vec{x} \land \varphi \land \psi \right),
        \]
        where
        \begin{align*}
            \varphi &= \bigwedge_{I\subseteq n} \left( \{x_i \mid i\notin I\}\probind_{\{x_i \mid i\in I\}\vec{z}}\{y_i\mid i\in I\} \right) \text{ and} \\
            \psi &= \bigwedge_{i<n}\left( y_i\probind_{\vec{x}\vec{z}}\{y_j \mid j\neq i\} \right).
        \end{align*}
\end{corollary}
\begin{proof}
    This follows by combining Propositions~\ref{Proposition: No-go theorems come from quantum mechanics} and~\ref{Proposition: Probabilistic realization implies possibilistic realization}, Corollary~\ref{Corollary: Probabilistic implies possibilistic} and Propositions~\ref{Proposition: GHZ} and~\ref{Proposition: Hardy}.
\end{proof}

It is shown in~\cite{ABRAMSKY20163} that every finite-dimensional tensor-product quantum-mechanical team which does not arise from a system whose state is merely a tensor product of 1-qubit states and maximally entangled 2-qubit states admits a Hardy-style proof of non-locality.

\section{Open Questions}
Questions left open include the following.
\begin{itemize}
    \item Do the concepts of downwards closedness and strong downwards closedness in probabilistic team semantics coincide?
    \item What properties commonly found in team-based logics, such as downwards closedness, do the operations $\PR$ and $\QR$ have?
    \item Does it make sense to think of $\PR$ as a ``modal'' operator? If yes, what axioms does it satisfy? How about $\QR$?
    \item Is the property of a probabilistic team being finite-dimensional tensor-product quantum-mechanical 
    definable in $\mathrm{ESOf}_{\R}$ or some similar logic?
    We can ask similar questions for the broader notions obtained by dropping finite dimensionality requirements.
    \item Is there a more general theorem behind Proposition~\ref{Proposition: Measurement locality + z-independence + locality implies the probabilistic versions}? Is there a formal reason why $\varphi\models\PR\varphi$ holds there while in Proposition~\ref{Proposition: Possibilistic no-signalling does not imply probabilistic no-signalling} it fails?
    \item Can the sheaf-theoretic framework of~\cite{Abramsky_2011} be translated to the language of team semantics in some reasonably satisfactory manner, allowing us to inspect more dependence and independence properties such as non-contextuali\-ty in terms of (a variant of) independence logic?
\end{itemize}

\bibliographystyle{plain}
\bibliography{main}

\end{document}